\documentclass[a4paper, twoside]{article}
\usepackage[english]{babel}

\usepackage[letterpaper,top=2cm,bottom=2cm,left=3cm,right=3cm,marginparwidth=1.75cm]{geometry}
\usepackage{fancyhdr}

\usepackage{amsmath}
\usepackage{amsthm}
\usepackage{amssymb}
\usepackage{enumitem}

\usepackage{algorithm}     
\usepackage{algpseudocode}

\usepackage{graphicx}
\usepackage[colorlinks=true, allcolors=blue]{hyperref}

\usepackage{tikz}
\usepackage{mathtools} % Add this to the preamble
\usepackage{subcaption}
\usepackage{amsmath}
\usepackage{tikz-3dplot} % For 3D coordinate transformations
\usepackage{pgfplots} 
\pgfplotsset{compat=1.18}
\usepackage{setspace}
\usetikzlibrary{arrows.meta,bending,positioning}
\newcommand{\interior}{\mathop{}\!\mathrm{int}}
\newcommand{\pblp}{\mathrm{PBLP}}
\newcommand{\tolp}{\mathrm{TOLP}}
\newcommand{\w}{\mathcal{W}}
\newcommand{\en}{\mathrm{EN}}
\newcommand{\ws}{\mathrm{WS}}
\newcommand{\wsc}{\mathrm{wsc}}

\newtheorem{theorem}{Theorem}
\newtheorem{corollary}[theorem]{Corollary}

\newtheorem{proposition}[theorem]{Proposition}
\theoremstyle{definition}
\newtheorem{definition}[theorem]{Definition}
\newtheorem{example}[theorem]{Example}
\newtheorem{remark}[theorem]{Remark}

\title{Parametric Biobjective Linear Programming\\[2ex]}
\author{Kezang Yuden\textsuperscript{1\thanks{Corresponding author: \texttt{yuden@mathematik.uni-kl.de};\\ Contributing authors: \texttt{l.nemesch@math.rptu.de}; \texttt{stefan.ruzika@math.rptu.de}}},  Levin Nemesch$^1$, and Stefan Ruzika$^1$}

\date{
	$^1$\small Department of Mathematics, RPTU Kaiserslautern-Landau,\\ 67663 Kaiserslautern, Germany\\[2ex]%%
}

\begin{document}
	\maketitle
	
	\begin{abstract}
		We consider  parametric linear programming problems with multiple objective functions depending linearly on some parameter. 
		Both parametric (single-objective) linear programming and (non-parametric) multi-objective linear programming are well-researched topics.
		However, literature on the combination of both, parametric linear programming with multiple objectives, is scarce.
		This research gap encourages our work in this field.
		Our main focus is on biobjective linear programs with a single parameter. We establish a connection of this problem to non-parametric multi-objective problems. Using the so-called weight set decomposition, we are able to explain the behavior of parametric biobjective linear programs when the parameter value is variated.
		We investigate two special cases of parametric biobjective linear programs: In the first, there is only one parametric objective and, in the second, the parametric dependency is the same for both objectives. We prove that there is a one-to-one correspondence between the solution of the parametric program and the solution of the triobjective program using the weighted sum scalarization. We provide structural insights to the solution of the parametric biobjective linear program with respect to extreme weights of the weight set of the triobjective linear program and develop solution strategies for the parametric program.\\

	\end{abstract}
	
	\section{Introduction}\label{sec1}
	
	The work presented in this article lies at the intersection of two well-studied branches of optimization and mathematical programming: \emph{parametric optimization} and \emph{multi-objective optimization}. Before describing the subject of our research, we first outline the development and current status of these two disciplines.
	
	Parametric optimization is a classical research topic in optimization. Its purpose is to study the behavior of solutions when the objective functions or constraints of an optimization problem depend on a non-negative parameter or multiple parameters.
	We call a problem \emph{linear parametric} if the objectives are affine linear dependent on the parameters and the constraints are independent of the parameters.
	%A single-parameter optimization problem arises when the problem varies with respect to the value of a single parameter.
	For a detailed overview of the literature on linear parametric optimization, we refer to the recent survey of Nemesch et al.~\cite{Nemesch2025}. In the 1950s, Gass and Saaty are among the first to explore the behavior of solutions when coefficients in the objective function are parameterized, introducing the concept of a parametric objective function~\cite{Saaty1954}. They present a technique for finding, for each parameter value, the optimal solution to the corresponding linear optimization problem~\cite{Gass1955}. Murty~\cite{Murty1976} adapts the parametric simplex method to solve linear parametric problems. Eisner and Severance~\cite{Eisner1976} develop an efficient algorithm to solve parametric optimization problems by computing parameter values where the optimal solution changes without exhaustively searching the entire parameter set. Later, Gusfield~\cite{Gusfield1983} develops a general method to find the sequence of optimal solutions for the entire parameter set. His algorithm builds upon the parametric search method described by Megiddo~\cite{Megiddo1978}. Several prominent combinatorial optimization problems have been studied in a parametric setting, including the parametric shortest problem~\cite{Young1991}, the parametric assignment problem~\cite{Gassner2010}, and the parametric minimum cost flow problem~\cite{Carstensen1983}.
	In addition to exact algorithms, several approximation algorithms have been developed for parametric optimization problems.
	For a comprehensive overview of approximation algorithms for parametric problems, we refer to the work of Nemesch et al.~\cite{Nemesch2025}. 
	%Giudici et al.~\cite{Giudici2017} present an approximation algorithm for the parametric version of the binary knapsack problem. 
	% More recently, Bazgan et al.~\cite{Bazgan2022} provide an approximation algorithm for a general class of parametric optimization problems which computes a set of solutions that approximate all feasible solutions for all parameter values within the given parameter interval.

	A second, well-studied field of optimization comprises multi-objective optimization problems (MOP), which deal with optimization of multiple, often conflicting objectives.
	%It has become a multidisciplinary field of research due to its practical relevance for real-world optimization problems. 
	Typically, no single feasible solution is optimal for all objectives simultaneously due to their conflicting nature. Instead, trade-offs must be made to balance the objectives. Thus, the notion of optimality is usually understood in the sense of \emph{Pareto optimality}. The Pareto optimal solutions represent the set of all optimal compromises and inform a decision maker about reasonable alternatives. The goal of multi-objective optimization is to compute the \emph{Pareto front}, or the set of \emph{nondominated images} in the image set.
	
	In 1975, Yu and Zeleny \cite{Yu1975} show that all nondominated images of a multi-objective linear problem (MOLP) with linear objective functions form a subset of the convex hull of the extreme nondominated images. 
	Benson~\cite{Benson1998} presents an algorithm, known as the ``Benson's Outer Approximation Algorithm'' to generate the set of all extreme nondominated images in the image set for a MOLP. Later, Heyde and Löhne~\cite{Heyde2008} present a geometric approach to duality for multi-objective linear programming, closely related to weighted sum scalarization. Another method for finding extreme nondominated images of multi-objective mixed integer programs is developed by Özpeynirci and Köksalan~\cite{Ozpeynirci2010}. Ehrgott et al.~\cite{Ehrgott2012} present a dual variant of Benson's outer approximation algorithm using the geometric duality theory of multi-objective linear program. 
	One of the well known methods to find the extreme nondominated images of MOLP with two objectives is the dichotomic approach~\cite{Cohon2004}.
	Most of the time, solving a multi-objective optimization problem  is a challenging task.
	For example, in multi-objective linear programming, the set of non-dominated images usually is infinite, and even the number of \emph{extreme non-dominated images} can be superpolynomial in the input size~\cite{Ruhe1988}.
	For a discussion on the hardness of general multi-objective combinatorial optimization problems, see~\cite{Boekler2017}. 
	
	A common approach to solve multi-objective optimization problems is by scalarization, which transforms the problem into a scalar-valued optimization problem. One such scalarization is the weighted sum scalarization method introduced by Zadeh~\cite{Zadeh1963}. A single objective is obtained by taking a non-negative linear combination of the objectives. This weighted sum objective is optimized over the same feasible set. Geoffrion~\cite{Geoffrion1968} proves that an optimal solution of the weighted sum scalarization of a MOP for some weight with strictly positive components is an efficient solution of a MOP. 
	In terms of the weight set, Benson and Sun \cite{Benson2002} develop a weight set decomposition algorithm to generate the extreme nondominated images of a MOLP. Przybylski et al.~\cite{Przybylski2010} present structural results of the weight set decomposition and propose an iterative algorithm to enumerate all extreme nondominated images by shrinking supersets of the actual weight set components. 
	More recently, Halffmann et al.~\cite{Halffmann2020} propose a weight set decomposition algorithm for multi-objective integer linear program.

	The combination of these two fields of optimization, i.\,e.~parametric and multi-objective optimization seem apparent, yet it remains largely unexplored in literature.
	The only closely related literature is that of sensitivity analysis, which typically examines how changes in objective coefficients affect the entire efficient set or a single solution.
	%Parametric optimization is connected to several other topics in optimization such as sensitivity analysis or robust optimization. 
	As data such as costs, risk, and demand might be uncertain in many real-life problems, sensitivity analysis has been proposed to address this data uncertainty. 
	%Over the years, several variants of parametric optimization problems such as 
	Wendell~\cite{Wendell1984} presents a tolerance approach in linear programming that caters to independent as well as simultaneous variation in several parameters in the coefficients of the objective function. This tolerance approach is further extended to multi-objective linear programs by Hansen et al.~\cite{Hansen1989}. They use the weighted sum method to find the maximum tolerance percentage for the weights to deviate so that a basic solution remains optimal. Sitarz~\cite{Sitarz2008} analyses changes associated to one objective function coefficient and proves that the parameter set in which a given solution remains efficient in a multi-objective linear program is convex. However, these studies examine the sensitivity region with respect to only one extreme efficient solution. Another closely related work is that of Andersen et al.~\cite{Andersen2025} where they study the sensitivity of the cost coefficients in multi-objective integer problems. They show that the sensitivity region with respect to a single objective function coefficient is a convex set. They obtain a sensitivity region in terms of permissible changes to coefficients so that a set of efficient solutions remains efficient. 
	
	This limited research motivates the study of theoretical and algorithmic approaches for linear parametric multi-objective programming (PMOP). Our article is an effort to lay the foundation in this area based on parametric and multi-objective optimization techniques. We consider two variants of the parametric biobjective linear program where the objective functions depend on a real-valued parameter. The goal of PMOP is to find, for each parameter value, a set of Pareto optimal solutions. More precisely, we aim for a decomposition of the parameter set into subintervals such that for every subinterval the same solutions are Pareto optimal. The parameter values where the Pareto optimal sets change are of particular interest. We develop two approaches to obtain these so-called breakpoints, and for each subinterval determined by two consecutive breakpoints, also a set of Pareto optimal solutions. To do so, we relate each of the two parametric biobjective programs to a triobjective program. In fact, each of the biobjective problems can be seen as a partial weighted sum scalarization \cite{Halffmann2020} of the corresponding triobjective optimization problem. We make use of the weight set decomposition and relate structures in the weight set to explain the optimal solution sets of the parametric biobjective programs.
	
	Our structural investigations result in two algorithms. The first algorithm uses any of the existing multi-objective programming algorithms and solves linear programs to determine breakpoints of the parametric problem. In the second algorithm, we adapt weighted sum decomposition algorithms. 
	
	This article is structured as follows: Definitions, terminology and structural results related to multi-objective and parametric optimization are presented in Section~\ref{sec2}.
	Then, we briefly introduce the weight set decomposition and explain the connection between its structure and the set of nondominated images of MOLP. Section~\refeq{sec3} considers two cases of the linear parametric biojective program and we present our results, establishing a connection between parametric biobjective problems and the corresponding triobjective problems. %We use the relation to the weighted sum scalarization of the triobjective problem to solve the linear parametric biobjective problems. 
	As a result, we present two algorithms in Section~\ref{sec4}: The \emph{Breakpoint Enumeration Algorithm} relies on a given set of extreme nondominated images of the corresponding triobjective linear program, and the \emph{Adapted Weight Set Decomposition} modifies existing weight set decomposition algorithms. 
	
	\section{Preliminaries}\label{sec2}
	
	We introduce basic definitions and some concepts of linear parametric programming as well as multi-objective programming. It is followed by some important results related to the weighted sum scalarization. In particular, we recapitulate results concerning the so-called weight set, i.\,e.~the structure of the set of weights needed for the weighted sum scalarization.
	
	We use the following variants of the componentwise ordering to compare objective function vectors and to establish a notion of optimality. For $y^1, y^2 \in \mathbb{R}^k$ with $k>1$, the \emph{weak componentwise order}, the \emph{componentwise order}, and the \emph{strict componentwise order} are defined by
	\begin{align*}
		y^1 \leqq y^2 &\iff y_i^1 \leq y_i^2 \text{ for all } i = 1, \dots , k, \\
		y^1 \leq y^2 &\iff y^1 \leqq y^2 \text{ but } y^1 \neq y^2,  \\
		y^1 < y^2 &\iff y_i^1 < y_i^2 \text{ for all } i = 1, \dots , k, 
	\end{align*}
	respectively. The non-negative orthant is denoted by $\mathbb{R}_\geqq^k \coloneqq \left \{x \in \mathbb{R}^k: x \geqq 0 \right\}$ and, likewise, the sets $\mathbb{R}_\geq^k$ and $\mathbb{R}_>^k$ are defined analogously using the componentwise and strict componentwise ordering. 
	
	\subsection{Parametric programming}
	
	We consider linear programming problems having an objective function which is subject to a non-negative parameter~$\lambda$.
	\begin{definition} [\textbf{Linear Parametric Linear Program}]
		Let $\lambda \in \mathbb{R}_\geq$ be a non-negative parameter. A linear parametric linear program is defined as 
		\begin{alignat*}{2}
			& \mathrm{min}\; \quad   &c^\top& x + \lambda \cdot d^\top x \\
			& \mathrm{s.\,t.} \quad &A&x\geqq b \tag{PLP}\\
			& \quad &&x \geqq 0, 
		\end{alignat*}
		where $c, d \in \mathbb{Q}^n$ are coefficient vectors, $A \in \mathbb{Q}^{m\times n}$, $m, n \in \mathbb{N}\setminus \{0\}$, $b \in \mathbb{Q}^m$, and the feasible set is denoted by $X \coloneqq \left \{x \in \mathbb{R}^{n}: Ax\geqq b, x \geqq 0\right\}$. 
	\end{definition} 
	
	%Solving a linear POP is understood as determining an optimal solution for all values of $\lambda \geq 0$. 
	%Basically, the goal is to compute an optimal solution in the form of a function of the parameter as defined below. 
	\begin{definition}
		The \emph{optimal value function} $P: \mathbb{R}_\geq \rightarrow \mathbb{R}$ maps each specific parameter value $\lambda^* \in \mathbb{R}_\geq$ to the optimal solution value $P(\lambda^*) \coloneqq c^\top x +  \lambda^* \cdot d^\top x$, where $x$ is an optimal solution for the (non-parametric) linear program obtained from the PLP for a fixed $\lambda^*$.
	\end{definition}
	
	%In other words, the optimal value function~$P(\lambda)$ maps each parameter value $\lambda \geq 0$ to the optimal objective value of the (non-parametric) linear optimization problem induced by $\lambda$.
	%\begin{proposition}[\cite{Gusfield1983}]\label{prop1}
	%	The function $P(\lambda)$ is piecewise linear and concave in $\lambda$.
	%\end{proposition}
	
	The function~$P(\lambda)$ can be obtained as the lower envelope of all linear functions associated with basic feasible solutions of the feasible set~$X$ (see Figure~\ref{fig1}). The function $P(\lambda)$ is piecewise linear and concave in $\lambda$ \cite{Gusfield1983}. The parameter values where the slope of $P(\lambda)$ changes are called \emph{breakpoints}.
	
	The solution to a PLP involves finding a set of optimal solutions for all values of $\lambda$. The ordered breakpoints of $P(\lambda)$ are denoted by $\lambda_{1} < \lambda_{2}  < \ldots  < \lambda_{k} $ for some $k \in \mathbb{N}$.
	Murty \cite{Murty1980} establishes that the number of intervals in the parameter set, and, as a result, the number of breakpoints in the parametric linear program can be exponential in the input size.
	
	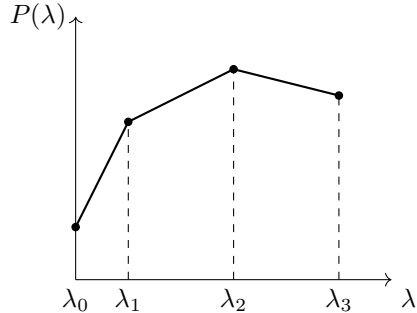
\begin{figure*}[h]
		\centering
		\begin{tikzpicture} [scale=0.7]
			\draw[->] node [below] {$\lambda_0$} (0,0) -- (6,0) node[below right] {$\lambda$};
			\draw[->] (0,0) -- (0,5) node[left]{$P(\lambda)$};
			\draw [black, thick]plot coordinates {(0,1) (1,3) (3,4) (5,3.5)};
			\draw [] plot[mark=*] coordinates {(0,1)};
			\draw [] plot[mark=*] coordinates {(1,3)};
			\draw [dashed] (1,0) node[below ] {$\lambda_1$} --(1,3);
			\draw [dashed] (3,0) node[below] {$\lambda_2$} --(3,4)  ;
			\draw [dashed]  (5,0) node[below] {$\lambda_3$}--(5,3.5);
			\draw []plot[mark=*] coordinates {(3,4)};
			\draw []plot[mark=*] coordinates {(5,3.5)};
		\end{tikzpicture}
		\caption{An illustration of an optimal value function $P(\lambda)$ with three breakpoints.\label{fig1}}
	\end{figure*}
	
	\subsection{Multi-objective optimization}
	We introduce some basic definitions and concepts of multi-objective optimization and polyhedral theory. For a detailed introduction to multi-objective optimization, we refer the reader to the book of Ehrgott~\cite{Ehrgott2005}. 
	\begin{definition} [\textbf{Multi-objective Linear Program}]
		A multi-objective linear program is defined as
		\begin{alignat*}{2}
			& \mathrm{min}\; \quad & C&x \\
			& \mathrm{s.\,t.} \quad  &A&x\geqq b \tag{MOLP} \\
			& \quad &&x \geqq 0, 
		\end{alignat*}
		where $C \in \mathbb{Q}^{k \times n}$ is the objective matrix consisting of the  rows $c_i, i = 1, \dots , k$, $A \in \mathbb{Q}^{m\times n}$, $m, n, k \in \mathbb{N}\setminus \{0\}$, $b \in \mathbb{Q}^m$. 
	\end{definition}
	We call the set $X = \left \{x \in \mathbb{R}^{n}: Ax\geqq b, x \geqq 0\right\}$ the feasible set and the set $Y \coloneqq \left \{Cx: x \in X\right\} \subseteq \mathbb{R}^k$ the image set of the MOLP. 
	For $k = 2$ and $k = 3$, a MOLP is called a biobjective linear program (BOLP) and a triobjective linear program (TOLP), respectively.
	
	As the objective functions of a MOLP are typically conflicting, there is usually no single optimal solution~$x$ that minimizes all individual objectives~$c_i x$ for all $i = 1, \dots, k$ simultaneously.
	
	%Based on these orderings, we use the concept of Pareto-optimality as the optimality concept for the MOLP. More precisely, we consider the following notion of optimality from Ehrgott \cite{Ehrgott2005}.
	
	\begin{definition} [\textbf{Efficiency and nondominance}]
		Let $x^* \in X$ be a feasible solution with corresponding image $y^* = Cx^*$. Then, $y^* \in Y$ is nondominated (weakly nondominated), if there is no other image $y \in Y$ such that $y \leq y^* (y < y^*)$. Correspondingly, $x^* \in X$ is efficient (weakly efficient), if $y^*$ is nondominated (weakly nondominated). 
	\end{definition}
	The set of nondominated (weakly nondominated) images is denoted by $Y_N$  ($Y_{wN}$) and the set of efficient (weakly efficient) solutions by $X_E$ ($X_{wE}$).
	For linear problems, the feasible set $X$ and also its image $Y = CX$ are both polyhedra. 
	We recall some basic concepts of polyhedral theory \cite{Ziegler2012}. The dimension of a polyhedron $P \subseteq \mathbb{R}^k$  is the maximum number of affinely independent points of $P$ minus one. A bounded polyhedron is called a polytope. For $w \in \mathbb{R}^k$ and $t\in \mathbb{R}$, the inequality $w^\top y \leq t$ is called valid for $P$ if $P \subseteq \left \{y \in \mathbb{R}^k : w^\top y \leq t \right\}$. A set $F \subset P$ is a face of $P$ if there is some valid inequality $w^ \top y \leq t$ such that $F = \left \{y \in P : w^\top y =t\right\}$. An extreme point is a face of dimension 0, an edge is a face of dimension 1 and a facet is a face of dimension $k-1$.
	
	In MOLP, the set of nondominated images is a convex hull defined by the \emph{extreme nondominated images}. An image $y \in Y$ is an extreme nondominated image if $y \in Y_{N}$ and $y$ is an extreme point of $Y$. We denote the set of extreme nondominated images by $Y_\en$. Consequently, the search for all nondominated images can be restricted to finding a set of efficient solutions that correspond to such extreme nondominated images.
	\begin{definition}
		A \emph{solution set} $S\subseteq X$ of a MOLP is a set such that for every $y \in Y_\en$ there is an $x \in S$ such that $Cx = y$.
		It is called \emph{minimal} if, additionally, there is no other solution set $S' \subseteq X$ with $\left| S' \right|< \left| S \right|$.
	\end{definition}
	
	\subsection{Weight set decomposition}
	Scalarization methods are useful techniques to compute nondominated images of a MOLP. A scalarization method transforms a multi-objective linear program into a single objective problem in a structured way by using additional parameters, constraints, or reference points. One commonly used scalarization method is the weighted sum scalarization. %It uses a conic combination of objective functions resulting in a scalar-valued single objective optimization problem.
	\begin{definition}
		Consider a MOLP and let a non-negative vector of weights~$w \in \mathbb R^k_\geq$ be given. The weighted sum scalarization problem of the MOLP with respect to weight~$w$ is the single objective linear program
		\begin{align*}
			\mathrm{\min_{x \in X}} \quad \ w^\top Cx \tag{$\ws(\mathrm{MOLP},w)$}.
		\end{align*}
	\end{definition}
	
	%Optimal solutions to $\ws(MOLP,w)_w$ are (weakly) efficient for the $MOLP$ if $ w \in { \mathbb{R}_>^q} (w \in { \mathbb{R}^q_\geq})$ and, vice versa, for every (weakly) efficient solution~$x$ of the $MOLP$, a suitable weight~$w$ can be chosen such that $x$ is optimal for $\ws(MOLP,w)_w$. This fact is precisely stated in the following theorem.
	%All optimal solutions of a weighted sum scalarization problem are efficient \cite{Geoffrion1968}. 
	A feasible solution that is optimal for some weighted sum scalarization problem is called \emph{supported} and a feasible solution whose image is uniquely optimal for some weighted sum scalarization problem is called \emph{extreme supported} \cite{Ehrgott2005}. 
	For linear programs, all extreme nondominated images are supported. Consequently, the set of extreme nondominated images of a MOLP, $Y_\en$ coincides with the set of extreme supported nondominated images.  
	
	\begin{theorem}[Isermann \cite{Isermann1974}]\label{thm1}
		A feasible solution $x^* \in X$ is an efficient solution of the MOLP if and only if there exists some $w \in \mathbb{R}_{>}^{k}$ such that $w^\top C x^* \leq w^\top Cx$ for all $x \in X$.
	\end{theorem}
	
	By Isermann's theorem, any efficient solution to MOLP can be found as an optimal solution of WS(MOLP,w) for some properly chosen weight $w$. Therefore, we are interested in the set of all weights and, also, its structure. 
	
	\begin{definition}
		The \emph{(normalized) weight set} is denoted by $\w$ and is defined as
		\begin{align*}
			\w \coloneqq \left \{w \in \mathbb{R}_{\geqq}^{k}: \sum_{i=1}^{k} w_i =1 \right\}.
		\end{align*}
	\end{definition}
	We consider the normalized weight set because normalization of the weight vectors does not affect the optimality of a solution obtained by weighted sum scalarization, i.\ e., $x$ is optimal for WS(MOLP,w) with $w \in \mathbb{R}^k_\geq$, if and only if $x$ is optimal for $w' \coloneqq \frac{1}{\sum_{i=1}^{k} w_i} w \in \w$.
	
	There exists a bijection between $\w$ and the weight set $\left \{w \in \mathbb{R}_{\geqq}^{k-1}: \sum_{i=1}^{k-1} w_i \leq1 \right\}$ and the set $\w$ is a polytope of dimension $k-1$.  So, in case of a weighted sum scalarization of a triobjective linear program, it is sufficient to consider the projection of $(w_1, w_2, w_3)$ onto $(w_1, w_2)$ because $w_3$ is uniquely determined by the values of $w_1$ and $w_2$, i.\ e. $w_3 = 1-w_1-w_2$.
	%and there is a bijection between $\w$ and the set $\left \{ \w \in \mathbb{R}_{\geq}^{q} : \sum_{i=1}^{q} w_i =1 \right\}$.
	A \emph{weight set decomposition} is the subdivision of the weight set into so-called \emph{weight set components}.
	All weights in a single weight set component map to the same nondominated image.
	For every $y \in Y $, the weight set component of $y$ is denoted by $ \w (y)$, and defined as 
	\begin{align*} 
		\w (y)\coloneqq \left \{ w \in \w : w^\top y = \mathrm{min} \left \{w^\top y' : y' \in Y\right\}\right\}.
	\end{align*}
	A weight set component $\mathcal W(y)$ consists of the subset of weights~$w \in \w$ for which the corresponding nondominated image~$y$ is an optimal value of $\ws(\mathrm{MOLP},w)$. A weight~$w$ is called an \emph{extreme weight} if it is an extreme point of $\w(y)$.  Important properties of these weight set components are presented by Przybylski et al.~\cite{Przybylski2010}. 
	
	\newcommand{\propitemref}[2]{\ref{#1}.\ref{#2}}
	\begin{proposition}[Przybylski et al. \cite{Przybylski2010}]\label{prop1}
		Let \( y \in Y_{N} \). Then, the following statements hold:
		\begin{enumerate}[label=\roman*.]
			\item	\(\w(y) = \left \{ w \in \w : w^\top y \leq w^\top y' \text{ for all } y' \in Y_\en \right\}\).  \label{prop1:partone} 
			\item	\(\w(y)\) is a convex polytope. \label{prop1:parttwo} 
			\item A nondominated image \( y \) is an extreme nondominated image of \( Y \) if and only if \(\w(y)\) has dimension \( k-1 \).\label{prop1:partthree} 
		\end{enumerate}
	\end{proposition}
	Proposition~\propitemref{prop1}{prop1:partone} implies that it is sufficient to consider only extreme nondominated images to define weight set components. Proposition~\propitemref{prop1}{prop1:partthree} implies that the weight set components corrposponding to the extreme nondominated images are full-dimensional polytopes in $\mathbb{R}_{\geqq}^{k-1}$.
	
	We observe from these properties that the structure of the weight set is closely linked to the structure of the nondominated image set of a MOLP by a one-to-one mapping between the weight set components and the extreme nondominated images.
	
	\section{Parametric biobjective linear programs}\label{sec3}
	In this section, we consider a parametric multi-objective linear program with two objective functions dependent on a single parameter $\lambda \in \mathbb{R}_\geq$.
	
	\begin{definition}[\textbf{Parametric Biobjective Linear Program}]
		Let $\lambda \in \mathbb{R}_\geq$ be a parameter. A parametric biobjective linear program is defined as
		\begin{alignat*}{2}
			& \mathrm{min}\; \quad &  C&x + \lambda Dx\\
			& \mathrm{s.\,t.} \quad  & A&x=b, \tag{PBLP} \\
			& \quad &&x \geqq 0, 
		\end{alignat*} 
		where $C, D \in \mathbb{Q}^ {2 \times n}$ consist of rows $c_i, d_i\ i=1,2$, $A \in \mathbb{Q}^{m\times n}$, $m, n \in \mathbb{N}\setminus \{0\}$, $b \in \mathbb{Q}^m$. Again, the feasible set is denoted by $X \coloneqq \left \{x \in \mathbb{R}^{n}: Ax=b, x \geqq 0\right\}$.
	\end{definition}
	
	%As mentioned in the introduction, parametric biobjective linear programs have not yet gained much attention. So, the motivation is to formulate a general PBLP and investigate different approaches to solve these problems. 
	As the literature frequently adopts a case-based approach, see \cite{Sitarz2008}, we focus on two special cases of the general PBLP with a parametric dependency in either one or both objectives. In the first case, only one of the two objectives depends linearly on the parameter. In the second case, we consider the same linear parametric dependency in both objectives. More precisely, we consider the following two special cases.\\
	Case I : One parametric objective
	\begin{alignat*}{2}
		& \mathrm{min}\; \quad &&	\begin{pmatrix}
			\, c_1 x + \lambda d_1 x\\
			\, c_2x
		\end{pmatrix}\\
		& \mathrm{s.\,t.} \quad  && x \in X. \tag{$\pblp^1$} 
	\end{alignat*}
	Case II :  Same parameter in both objectives
	\begin{alignat*}{2}
		& \mathrm{min}\; \quad && 	\begin{pmatrix}
			\, c_1 x+ \lambda d_1 x\\
			\, c_2 x+ \lambda d_1 x
		\end{pmatrix}\\
		& \mathrm{s.\,t.} \quad  && x \in X. \tag{$\pblp^2$}
	\end{alignat*} 
	
	%For each $\lambda \in \mathbb{R}_\geq$ we solve the corresponding biobjective optimization problem and find the set of efficient solutions which represents the nondominated images in the image set. 
	
	For notational consistency and to avoid repetition, we adopt the convention of denoting both cases of PBLP together as $\pblp^j$ for $j = 1, 2$ collectively when the results hold for both problems.
	%We consider the set of efficient solutions for our problems $\pblp^j$ instead of the nondominated image set because the two-dimensional nondominated image set of $mathrm{PBLP}^j(\lambda)$ is constantly changing with respect to $\lambda$.
	% Therefore, the solution to $mathrm{PBLP}^j$ involves solving the corresponding non-parametric biobjective linear program for each $\lambda \in \mathbb{R}_{\geq}$ and identifying a corresponding set of efficient solutions. 
	
	The problem $\pblp^j$, for some fixed value of $\lambda$, is a non-parametric biobjective linear program, and we denote it by $\pblp^j(\lambda)$. Furthermore, we denote the set of extreme non-dominated images of $\pblp^j(\lambda)$ by $Y_\en(\pblp^j(\lambda))$ and a corresponding minimal solution set by $S(\pblp^j(\lambda))$. 
	With some abuse of notation, we write $S(\pblp^j(\lambda_1))=S(\pblp^j(\lambda_2))$ if $\pblp^j(\lambda_1)$ and $\pblp^j(\lambda_2)$ share a minimal solution set, and $S(\pblp^j(\lambda_1))\neq S(\pblp^j(\lambda_2))$ if they do not.
	Our analysis relates to a minimal solution set that corresponds to extreme nondominated images of $\pblp^j(\lambda)$. Therefore, a minimal solution set of $\pblp^j$ equates to a set that contains optimal solutions of $\pblp^j(\lambda)$ for each fixed value of $\lambda \geq 0$. We denote this set by $S$. 
	
	\begin{definition}
		A \emph{solution set} $S \subseteq X$ of $\pblp^j$ is a set such that for every $\lambda \geq 0$, $S$ contains, as a subset, a solution set for the BOLP $\pblp^j(\lambda)$.
		It is called \emph{minimal} if, additionally, there is no other solution set $S' \subseteq X$ for $\pblp^j$ with $\left| S' \right|< \left| S \right|$.
	\end{definition}
	
	We will relate $\pblp^j$ to the corresponding triobjective linear program using the weighted sum scalarization method. To this end, we consider the triobjective linear problem with the objective functions $c_1 x, c_2 x$, and $d_1 x$, i.\ e. 
	\begin{alignat*}{2}
		& \mathrm{\min}\; \quad && \left( c_1 x , c_2 x, d_1 x \right) ^\top \\
		& \mathrm{s.\ t.} \quad && x \in X, \tag{TOLP}
	\end{alignat*}
	and denote its set of extreme nondominated images by $Y_\en(\tolp)$. We define the weight set $\w(\tolp)$ of TOLP as
	\begin{align*}
		\w(\tolp) \coloneqq \left\{w^* \in \mathbb{R}_{\geqq}^{3}: \sum_{i=1}^{3} w^*_i =1 \right\}.
	\end{align*}
	Note that we consider the projection of $w^* \in \mathbb{R}_{\geqq}^{3}$ in $\mathbb{R}_{\geqq}^{2}$ if necessary, and both weight sets, whether its dimension is 3 or 2, are denoted by $\w(\tolp)$. 
	\noindent The weighted sum scalarization of $\tolp$ with a normalized weight $w^* \in \w(\tolp)$ is 
	\begin{align*} 
		\mathrm{\min_{x \in X}}\ w^*_1 c_1 x + w^*_2c_2 x + w^*_3 d_1 x.
		\tag{$\ws(\tolp, w^*)$}
	\end{align*}
	%We will be using this weighted sum scalarization in our analysis in both the cases.

	\subsection{Case I : One parametric objective}
	
	We approach the problem by applying the weighted sum scalarization to the parametric biobjective linear program $\pblp^1$ and formally characterize its relationship to the corresponding triobjective linear program in our results. 
	
	\noindent For a fixed value of $\lambda$, the weighted sum scalarization of $\pblp^1$ is 
	\begin{align*} 
		\mathrm{\min_{x \in X}}\ w_1 (c_1 x + \lambda d_1 x) + w_2 c_2 x \tag{$\ws(\pblp^1(\lambda), w)$}
	\end{align*}
	\noindent where $w \in\mathbb{R}_{\geqq}^{2}$ and $w_1 +w_2 = 1$. \\
	We reformulate this problem using $w_2 = 1- w_1$ and obtain
	\begin{align*} 
		\mathrm{\min_{x \in X}}\ w_1 c_1 x  + (1-w_1) c_2 x + w_1 \lambda d_1 x.
	\end{align*}
	
	\noindent The problem can be interpreted as a weighted sum scalarization problem of the TOLP with weight vector $\left(w_1, 1-w_1, w_1 \lambda\right)$.
	It holds that 
	\begin{align*} 
		w_1 + w_1 \lambda + (1-w_1) = 1 + w_1 \lambda \geq 1.
	\end{align*}
	We normalize the weights and get
	\begin{align} 	\label{eqn1}
		\mathrm{\min_{x \in X}}\ \frac{w_1}{1+ w_1 \lambda}(c_1 x)  + \frac{w_2}{1+ w_1 \lambda} (c_2 x)+ \frac{w_1 \lambda }{1+ w_1 \lambda}(d_1 x),
	\end{align}
	which is a particular case of $\ws(\tolp, w^*)$ with the corresponding weights
	\begin{align*}
		w^* = \left( \frac{w_1}{1+w_1 \lambda},  \frac{w_2}{1+w_1 \lambda},  \frac{w_1 \lambda}{1+w_1 \lambda}\right)^\top.
	\end{align*}
	
	%\subsubsection{Structural results for $\pblp^1$}
	We now present our main result, which establishes the equivalence between the efficient solutions of the problem $\pblp^1$ and the efficient solutions of the corresponding triobjective linear program.
	% This result is formalized in the following theorem.
	\begin{theorem}\label{thm2}
		A feasible solution $x^*$ is optimal for $\ws(\tolp, w^*)$ with non-negative weights $w_1^*$, $w_2^*$, $w_3^*$, where $w_1^* > 0$ if and only if there exist a parameter $\lambda \geq 0$ and non-negative weights $w_1$, $w_2$ where $w_1 > 0$ such that $x^*$ is optimal for $\ws(\pblp^1(\lambda), w)$.
	\end{theorem}
	\begin{proof}
		We first show that $x^*$ is optimal for $\ws(\tolp, w^*)$, implying that there exist $\lambda \geq 0$ and non-negative weights $w_1$, $w_2$ such that $x^*$ is optimal for $\ws(\pblp^1(\lambda), w)$.  Let $x^*$ be optimal for $\ws(\tolp, w^*)$. \\
		We define
		\begin{align}
			\lambda & \coloneqq    \frac{w^*_3}{w^*_1}\nonumber\\
			w_1 & \coloneqq w^*_1 	\label{eqn2}\\
			w_2 & \coloneqq w^*_2.\nonumber
		\end{align}
		Suppose $x^*$ is not optimal for $\ws(\pblp^1(\lambda), w)$. Then there exists some $x'$ that is feasible for $\ws(\pblp^1(\lambda), w)$ where $x' \neq x^*$ such that
		\begin{align*}
			w_1 c_1 x' + w_2 c_2 x'+w_1 \lambda d_1 x' & < w_1 c_1 x^* + w_2 c_2 x^* + w_1 \lambda  d_1 x^*.\\
			\intertext{We plug in (\ref{eqn2}) to get}
			w^*_1 c_1 x' + w^*_2 c_2 x' +w^*_1 ( \frac{w^*_3}{w^*_1} d_1 x') &<  w^*_1 c_1 x^* + w^*_2 c_2 x^* + w^*_1( \frac{w^*_3}{w^*_1} )d_1 x^*.\\
			\intertext{This is equivalent to}
			w^*_1 c_1 x' + w^*_2 c_2 x' + w^*_3 d_1 x' &< w^*_1 c_1 x^* + w^*_2  c_2 x^* + w^*_3 d_1 x^*.
		\end{align*}
		This leads to a contradiction to $x^*$ being optimal for $\ws(\tolp, w)$. \\
		Conversely, let $x^*$ be optimal for $\ws(\pblp^1(\lambda), w)$ with non-negative weights $w_1$, $w_2$ and for some non-negative $\lambda$. \\
		We define  $w_1^*$, $w_2^*$, $w_3^*$ using (\ref{eqn2}),
		\begin{align}
			w^*_1 & \coloneqq w_1 \nonumber\\
			w^*_2 & \coloneqq w_2 \label{eqn3}\\
			w^*_3 & \coloneqq w_1 \lambda. \nonumber
		\end{align}
		Suppose $x^*$ is not optimal for $\ws(\tolp, w^*)$. Then there exists $x'$ that is feasible for $\ws(\tolp, w^*)$ where $x' \neq x^*$ such that
		\begin{align*} 
			w^*_1 c_1 x' + w^*_2 c_2 x' + w^*_3 d_1 x' &< w^*_1 c_1 x^* + w^*_2  c_2 x^* + w^*_3 d_1 x^*.\\
			\intertext{We plug in (\ref{eqn3}) to get}
			w_1 c_1 x' + w_2 c_2 x' +  w_1 \lambda d_1 x' &<  w_1 c_1 x^* + w_2 c_2 x^* +  w_1 \lambda d_1 x^*.
		\end{align*}
		This leads to a contradiction that $x^*$ is optimal for $\ws(\pblp^1(\lambda), w)$.
	\end{proof}
	
	Observe that we assume $w_1^* > 0$ and $w_1 >0$ in Theorem~\ref{thm2}.
	Based on the construction in the proof, the parameter $\lambda$ is undefined if $w_1^* =0$, and $w_1^*$ approaching $0$ is equivalent to the parameter $\lambda$ approaching $\infty$.
	And indeed, the one-to-one correspondence from Theorem~\ref{thm2} does not hold if $w_1^* =0$:
	there exist $w^*\in\mathbb{R}^3_\geqq$ with $w^*_1=0$ and an optimal solution $x^*$ for  $\ws(\tolp,w^*)$ such that there is no $\lambda\geq0$ and $w\in\mathbb{R}^2_\geqq$ where $x^*$ is optimal for $\ws(\pblp^1(\lambda),w)$.
	This can be illustrated by the following example: let $(1,1,1)$, $(0,1,1)$ and $(0,0,2)$ be the extreme points of $X$, and let $c_1 x = x_1$, $c_2 x=x_2$ and $d_1 x= x_3$.
	For $w^*=(0,0,1)$, both $(1,1,1)$ and $(0,1,1)$ are optimal solutions of $\ws(\tolp,w^*)$.
	But for every choice of $\lambda\geq0$ and $w\in\mathbb{R}^2_\geqq$ with $w_1>0$, $(0,1,1)$ is a better solution than $(1,1,1)$ for $\ws(\pblp^1(\lambda),w)$.
	And for $w_1=0$, the solution $(0,0,2)$ is optimal for $\ws(\pblp^1(\lambda),w)$.
	Therefore, $(1,1,1)$ is never optimal for $\ws(\pblp^1(\lambda),w)$.\\
	In the converse case, for any given parameter value $\lambda \geq 0$ and a weight $w\in\mathbb{R}^2_\geqq$ with $w_1=0$ for $\ws(\pblp^1(\lambda), w)$, we can see that it is possible to construct a weight $w^* \in \mathbb{R}^3_\geqq$ for $\ws(\tolp, w^*)$ using (\ref{eqn3}). 
	In fact, an optimal solution $x$ for $\ws(\pblp^1(\lambda), w)$ with a weight $w\in\mathbb{R}^2_\geqq$ such that $w_1=0$ remains optimal for $\ws(\tolp, w^*)$ with the weight $w^*=(0,w_2,0)$ because $w^*_1=0$ and $w^*_3 = 0$ by construction.
	
	For every fixed value of $\lambda$, $\pblp^1(\lambda)$ is a biobjective linear program.
	Its weight set is a one-dimensional polytope. However, considering the parameter $\lambda$, we can extend this interpretation of the set of the one-dimensional weight set to a higher dimension by a mapping $\w(\pblp^1( \lambda))$, which is defined below. This extended representation enables us to derive several important theoretical results.
	
	\begin{definition}
		For a given $\lambda$, we define $\w(\pblp^1 (\lambda))$ as,
		\begin{align*}
			\w(\pblp^1 (\lambda)) \coloneqq \left \{\left(\frac{w_1}{1+w_1 \lambda}, \frac{w_2}{1+w_1 \lambda}, \frac{w_1 \lambda}{1+w_1 \lambda}\right): (w_1, w_2) \in \mathbb{R}_{\geqq}^{2}, w_1+w_2 = 1\right\}.
		\end{align*}
	\end{definition}
	
	\begin{proposition}\label{prop2}
		Let $\w(\pblp^1 (\lambda))$ and $\w(\tolp)$ be weight sets of $\pblp^1$ for a fixed value $\lambda$ and of $\tolp$, respectively. Then,it holds
		\begin{align*}
			\w(\pblp^1 (\lambda)) \subsetneq \w(\tolp).
		\end{align*}
	\end{proposition}
	\begin{proof}
		Let $w^* \in \w(\pblp^1 (\lambda))$. Then, by the definition of $\w(\pblp^1 (\lambda))$, it is
		\begin{align*}
			w^* = \left(\frac{w_1}{1+w_1 \lambda}, \frac{w_2}{1+w_1 \lambda}, \frac{w_1 \lambda}{1+w_1 \lambda}\right).
		\end{align*} 
		The sum of the components satisfies
		\begin{align*}
			\frac{w_1}{1+w_1 \lambda} + \frac{w_2}{1+w_1 \lambda} + \frac{w_1 \lambda}{1+w_1 \lambda} = \frac{1+ w_1 \lambda}{1+w_1 \lambda} = 1,
		\end{align*} 
		and, thus $w^* \in \w(\tolp)$.\\
		To show that $\w(\pblp^1 (\lambda))$ is a proper subset of $\w(\tolp)$, we consider the particular weight $w' \coloneqq (0,0,1) \in \w(\tolp)$. Since for any given $\lambda$, the third component of any $w^* \in \w(\pblp^1 (\lambda))$ satisfies
		\begin{align*}
			\frac{w_1 \lambda}{1+w_1 \lambda} \neq 1.
		\end{align*}
		Hence, weight $w' \notin \w(\pblp^1 (\lambda))$.
	\end{proof}
	
	\begin{proposition}\label{prop3}
		For two distinct parameter values $\lambda_1 < \infty$ and $\lambda_2 < \infty$, $\lambda_1 \neq \lambda_2$, it holds that
		\begin{align*}
			\w(\pblp^1(\lambda_1)) \cap \w(\pblp^1(\lambda_2)) = \{(0,1, 0)\}.
		\end{align*}
	\end{proposition}
	\begin{proof}
		By the definition of $\w(\pblp^1 (\lambda))$, we have 
		\begin{align*}
			\w(\pblp^1 (\lambda_1)) \coloneqq \left \{\left(\frac{w_1}{1+w_1 \lambda_1}, \frac{w_2}{1+w_1 \lambda_1}, \frac{w_1 \lambda_1}{1+w_1 \lambda_1}\right):  (w_1, w_2) \in \mathbb{R}_{\geq}^{2}, w_1+w_2 = 1 \right\}
			\intertext{and}
			\w(\pblp^1 (\lambda_2)) \coloneqq \left \{\left(\frac{w_1}{1+w_1 \lambda_2}, \frac{w_2}{1+w_1 \lambda_2}, \frac{w_1 \lambda_2}{1+w_1 \lambda_2 }\right):  (w_1, w_2) \in \mathbb{R}_{\geq}^{2}, w_1+w_2 = 1\right\}.
		\end{align*}
		We first show that $(0,1, 0) \in \w(\pblp^1(\lambda_1))$ and $(0,1, 0) \in \w(\pblp^1(\lambda_2))$.\\
		If $w_1 = 0$, the corresponding mapping in $\w(\pblp^1(\lambda_1))$ is
		\begin{align*}
			\left(\frac{0}{1+0}, \frac{w_2}{1+0}, \frac{0}{1+0}\right) &= (0,1,0) \in \w(\pblp^1 (\lambda_1)) 
			\intertext{and the corresponding mapping in $\w(\pblp^1(\lambda_2))$ is}
			\left(\frac{0}{1+0}, \frac{w_2}{1+0}, \frac{0}{1+0}\right) &= (0,1,0) \in \w(\pblp^1 (\lambda_2)).
		\end{align*}
		However, for all $w_1 >0$, if we compare each corresponding pair of individual components of $w^* \in \w(\pblp^1 (\lambda_1))$ and $w' \in \w(\pblp^1 (\lambda_2))$, we get 
		\begin{align*}
			\frac{w_1}{1+w_1 \lambda_1} &\neq \frac{w_1}{1+w_1 \lambda_2}\\
			\frac{w_2}{1+w_1 \lambda_1} &\neq \frac{w_2}{1+w_1 \lambda_2}\\
			\frac{w_1 \lambda_1}{1+w_1 \lambda_1} &\neq \frac{w_1 \lambda_2}{1+w_1 \lambda_2}
		\end{align*}
		because $\lambda_1 \neq \lambda_2$. Thus, 
		\begin{align*}
			\w(\pblp^1(\lambda_1)) \cap \w(\pblp^1(\lambda_2)) = \{(0,1, 0)\}.
		\end{align*}
		
	\end{proof}
	
	\begin{figure*}[b]
		\centering
		\begin{tikzpicture}[scale=5]
			
			% Define triangle vertices
			\coordinate (A) at (1,0);
			\coordinate (B) at (0,1);
			\coordinate (C) at (0,0);
			\coordinate (D) at (0.45,0);
			\coordinate (E) at (0.33,0.3);
			
			% Draw the triangle
			\draw[thick] (A) -- (B) -- (C) -- cycle;
			
			% Draw 10 lines starting from (0,1) and ending at x-axis within triangle
			\foreach \i in {4,5,...,9} {
				\pgfmathsetmacro{\xend}{0.1 + \i*0.09} % x-intercept from 0.1 to 1.0
				\pgfmathsetmacro{\opacity}{0 + \i*0.08} % Increasing opacity for fading effect
				
				% Draw the line from (0,1) to (xend,0)
				\draw[opacity=\opacity, magenta] (0,1) -- (\xend,0);
			}
			% Optional: Add labels for the triangle vertices
			\node[right] at (A) {$w_1^*$};
			\node[above ] at (B) {$w_2^*$};
			\node[below] at (D) {$\lambda$};
			\node[left] at (E) {\tiny $\color{magenta} 	\mathcal{L}_{\w(\pblp^1(\lambda))} $};
		\end{tikzpicture}
		\caption{An illustration of \textbf{$	\mathcal{L}_{\w(\pblp^1(\lambda))} $} for different values of $\lambda$ in the weight set of TOLP.\label{fig2}}
	\end{figure*}
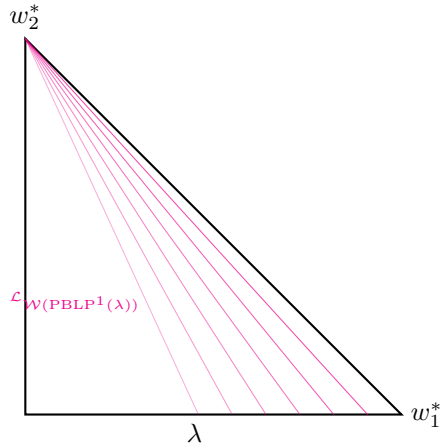
	
	\begin{remark} \label{remark1}
		Let $\w(\pblp^1 (\lambda))$ be the weight set of $\pblp^1(\lambda)$ and $\w(\tolp)$ be the weight set of the corresponding TOLP. Then, as a consequence of equating the weights in $\w(\pblp^1(\lambda))$ for every $\lambda$ with $\w(\tolp)$, we observe the following:
		\begin{enumerate}[label=\roman*.]
			\item \label{remark1:partone}
			$$\bigcup\limits_{\lambda \geq 0}{}   \w(\pblp^1(\lambda)) = \{ w^* \in \w(\tolp): w^*_1 \neq 0\} \cup \{(0,1,0)\}$$
			\item \label{remark1:parttwo} The projection of the weight set $\w(\pblp^1 (\lambda))$ in $\mathbb{R}^2$ can be defined as the line segment, 
			\begin{align*}
				\mathcal{L}_{\w(\pblp^1(\lambda))} = \left\{ (w^*_1, w^*_2) \mid w^* \in \w(\tolp),\ w^*_1(1 + \lambda) + w^*_2 = 1 \right\}
			\end{align*}
			\item \label{remark1:partthree} The slope $m_\lambda$ of the line segment $	\mathcal{L}_{\w(\pblp^1(\lambda))} $ is given by
			\begin{align*}
				m_\lambda = -(1+\lambda).
			\end{align*}
			\item \label{remark1:partfour}  The end points of $	\mathcal{L}_{\w(\pblp^1(\lambda))} $ are $(0, 1)$ and $\left(\frac{1}{1+ \lambda}, 0\right)$.
		\end{enumerate}
	\end{remark} 
	
	Remark~\propitemref{remark1}{remark1:partone} implies that if we vary $\lambda$ in $\pblp^1$ and project all the corresponding weight sets $\w(\pblp^1(\lambda))$ in $\mathbb{R}^2$, we obtain a nearly complete weight set decomposition of the associated TOLP except for the weights with $w_1^* =0$ due to the reason mentioned in Theorem~\ref{thm2}.
	Remark~\propitemref{remark1}{remark1:parttwo}, \propitemref{remark1}{remark1:partthree} and \propitemref{remark1}{remark1:partfour} are direct results of the definition of $\w(\pblp^1 (\lambda))$ for $w_1 \in \left[0,1\right]$ and are illustrated in Figure~\ref{fig2}. Note that we use the projection of $w^* \in \w(\tolp, w^*)$ onto $(w^*_1, w^*_2) \in \mathbb{R}_\geqq^2$ in our analysis.
	
	\subsection{Case II: Same parametric dependency on both objectives }
	
	We now consider the problem $\pblp^2$ having a non-negative parameter $\lambda$ in both objectives.\\
	\noindent For a fixed value of $\lambda$, the weighted sum scalarization of $\pblp^2$ is 
	\begin{align*} 
		\mathrm{\min_{x \in X}}\ w_1 (c_1 x + \lambda d_1 x) + w_2 (c_2 x+ \lambda d_1 x)
		\tag{$\ws(\pblp^2(\lambda), w)$}
	\end{align*}
	\noindent where $w \in \mathbb{R}_\geqq^{2}$ and $w_1 + w_2 = 1$.\\
	We reformulate this problem and obtain
	\begin{align*} 
		&\mathrm{\min_{x \in X}}\ w_1 c_1 x + w_2 c_2 x + (w_1+ w_2)  \lambda d_1 x .\\
		\intertext{Using $w_2 = 1- w_1$, this is equivalent to}
		&\mathrm{\min_{x \in X}}\ w_1 c_1 x + (1-w_1) c_2 x + \lambda d_1 x . 
	\end{align*}
	\noindent It can be interpreted as a weighted sum scalarization problem of the TOLP with a weight vector $(w_1, 1-w_1,\lambda)$. 
	Since $w_1 + w_2 = 1$, it holds that
	\begin{align*} 
		w_1 + w_2 + \lambda  = 1 +  \lambda \geq 1.
	\end{align*}
	\noindent We normalize the weights and get
	\begin{align} 	\label{eqn4}
		\mathrm{\min_{x \in X}}\ \frac{w_1}{1+\lambda} c_1 x + \frac{w_2 }{1+\lambda}c_2 x + \frac{\lambda}{1+\lambda} d_1 x
	\end{align}
	which is a particular case of $\ws(\tolp, w^*)$ with the corresponding weights
	\begin{align*}
		w^* = \left(\frac{w_1}{1+\lambda}, \frac{w_2}{1+\lambda}, \frac{\lambda}{1+\lambda} \right).
	\end{align*}
	
	%\subsubsection{Structural results for $\pblp^2$}
	We establish the equivalence between the optimal solutions of the weighted sum scalarization problem of $\pblp^2$ and the weighted sum scalarization problem of the corresponding TOLP in the following theorem.
	\begin{theorem}\label{thm3}
		A feasible solution $x^*$ is optimal for $\ws(\tolp, w^*)$ with non-negative weights $w_1^*$, $w_2^*$, $w_3^*$, where $w_3^* \neq 1$ if and only if there exists a parameter value $\lambda \geq 0$ and non-negative weights $w_1$, $w_2$, such that $x^*$ is optimal for $\ws(\pblp^2(\lambda), w)$.
	\end{theorem}
	\begin{proof}
		This proof is analogous to the proof of Theorem~\ref{thm2} where we define
		\begin{align} \label{eqn5}
			w_1 & \coloneqq \frac{w^*_1}{1-w^*_3} \nonumber\\
			w_2 & \coloneqq \frac{w^*_2}{1-w^*_3} \\
			\lambda & \coloneqq \frac{w^*_3}{1-w^*_3},\nonumber
		\end{align}
		in the first part of the proof and we define
		\begin{align*}
			w^*_1 & \coloneqq w_1 \nonumber\\
			w^*_2 & \coloneqq w_2\nonumber\\
			w^*_3 & \coloneqq \lambda \nonumber
		\end{align*}
		in the converse case.
	\end{proof}
	
	We now address the case $w_3^* = 1$ as we assume $w_3^* \neq 1$ in the theorem above. 
	Based on the construction of the proof, the parameter $\lambda$ is undefined if $w_3^* = 1$. 
	The limit of $w_3^*$ approaching $1$ is equivalent to the parameter $\lambda$ approaching $\infty$ in the weighted sum scalarization of $\pblp^2$. 
	Additionally, the one-to-one correspondence of Theorem~\ref{thm3} does not hold: there exist $w^*\in\mathbb{R}^3_\geq$ with $w^*_3=1$ and an optimal solution $x^*$ for $\ws(\tolp,w^*)$ such that there are no $\lambda\geq0$ and $w\in\mathbb{R}^2_\geq$ where $x^*$ is optimal for $\ws(\pblp^2(\lambda),w)$.
	This can be illustrated by the following example: let $(1,1,1)$ and $(0,0,1)$ be the extreme points of $X$, and let $c_1 x = x_1$, $c_2 x=x_2$ and $d_1 x= x_3$.
	For $w^*=(0,0,1)$, both $(1,1,1)$ and $(0,0,1)$ are optimal solutions of $\ws(\tolp,w^*)$.
	But for every choice of $\lambda\geq0$ and $w \in\mathbb{R}^2_\geq$, $(0,0,1)$ is a better solution than $(1,1,1)$ for $\ws(\pblp^2(\lambda),w)$.
	Therefore, $(1,1,1)$ is never optimal for $\ws(\pblp^2(\lambda),w)$.

	We extend the weight set of $\pblp^2( \lambda)$ to a higher dimension the mapping $\w(\pblp^2( \lambda))$ just as in case I. 
	
	\begin{definition}
		For a fixed value of $\lambda$, we define $\w(\pblp^2 (\lambda))$ as,
		\begin{align*}
			\w(\pblp^2 (\lambda)) \coloneqq \left \{\left(\frac{w_1}{1+\lambda}, \frac{w_2}{1+\lambda}, \frac{\lambda}{1+\lambda}\right): (w_1, w_2) \in \mathbb{R}_{\geqq}^{2}, w_1+w_2 = 1\right\}.
		\end{align*}
	\end{definition}
	
	\begin{proposition}\label{prop4}
		Let $\w(\pblp^2 (\lambda))$ and $\w(\tolp)$ be weight sets of $\pblp^2$ for a fixed $\lambda$ and the corresponding TOLP, respectively, then 
		\begin{align*}
			\w(\pblp^2 (\lambda)) \subsetneq \w(\tolp).
		\end{align*}
	\end{proposition}
	\begin{proof}
		Let $w^* \in \w(\pblp^2 (\lambda))$. Then by the definition of $\w(\pblp^2 (\lambda))$, 
		\begin{align*}
			w^* = \left(\frac{w_1}{1+\lambda}, \frac{w_2}{1+\lambda}, \frac{ \lambda}{1+\lambda}\right).
		\end{align*} 
		Since the sum of the components
		\begin{align*}
			\frac{w_1}{1+\lambda} + \frac{w_2}{1+\lambda} + \frac{ \lambda}{1+\lambda} = 1 .
		\end{align*} 
		Thus, $w^* \in \w(\tolp)$.\\
		Consider a particular weight $w^* = (0,0,1) \in \w(\tolp)$. For any given value of $\lambda > 0$ the third component of a weight in $\w(\pblp^2 (\lambda))$, 
		\begin{align*}
			\frac{\lambda}{1+\lambda} \neq 1 .
		\end{align*}
		Therefore, it holds $w^* \notin \w(\pblp^2 (\lambda))$.
	\end{proof}
	
	\begin{proposition}\label{prop5}
		For two parameter values $\lambda_1$ and $\lambda_2$, if $\lambda_1 \neq \lambda_2$ then
		\begin{align*}
			\w(\pblp^2 (\lambda_1)) \cap \w(\pblp^2 (\lambda_2)) = \emptyset.
		\end{align*}
	\end{proposition}
	\begin{proof}
		By the definition of $\w(\pblp^2 (\lambda))$, it it
		\begin{align*}
			\w(\pblp^2 (\lambda_1)) &\coloneqq \left \{\left(\frac{w_1}{1+\lambda_1}, \frac{w_2}{1+\lambda_1}, \frac{\lambda_1}{1+\lambda_1}\right): (w_1, w_2) \in \mathbb{R}_{\geq}^{2}, w_1+w_2 = 1\right\}
			\intertext{and}
			\w(\pblp^2 (\lambda_2)) &\coloneqq \left \{\left(\frac{w_1}{1+\lambda_2}, \frac{w_2}{1+\lambda_2}, \frac{\lambda_2}{1+\lambda_2}\right):(w_1, w_2) \in \mathbb{R}_{\geq}^{2}, w_1+w_2 = 1\right\}.
		\end{align*}
		For all $w_1 \geq 0$, by comparing each corresponding pair of individual components of $w \in \w(\pblp^2 (\lambda_1))$ and $w' \in \w(\pblp^2 (\lambda_1))$, we get 
		\begin{align*}
			\frac{w_1}{1+w_1 \lambda_1} &\neq \frac{w_1}{1+w_1 \lambda_2}\\
			\frac{w_2}{1+w_1 \lambda_1} &\neq \frac{w_2}{1+w_1 \lambda_2}\\
			\frac{\lambda_1}{1+\lambda_1} &\neq \frac{\lambda_2}{1+\lambda_2}
		\end{align*}
		since $\lambda_1 \neq \lambda_2$. Thus,
		\begin{align*}
			\w(\pblp^2 (\lambda_1)) \cap \w(\pblp^2 (\lambda_2)) = \emptyset.
		\end{align*}
	\end{proof}
	
	\begin{remark} \label{remark2}
		Let $\w(\pblp^2 (\lambda))$ be the weight set of $\pblp^2$ for a fixed $\lambda$ and let $\w(\tolp)$ be the weight set of the corresponding TOLP. Then, as a consequence of equating the weights in $\w(\pblp^2(\lambda))$ for every $\lambda$ with $\w(\tolp)$, we observe the following:
		\begin{enumerate}[label=\roman*.]
			\item \label{remark2:partone}
			$$	\bigcup\limits_{\lambda \geq 0}{}   \w(\pblp^2 (\lambda)) = \w(\tolp) \setminus \left \{(0,0,1)\right\} $$
			\item \label{remark2:parttwo} The projection of the weight set $\w(\pblp^2 (\lambda))$ in $\mathbb{R}^2$ can be defined by a line segment,
			\begin{align*}
				\mathcal{L}_{\w(\pblp^2(\lambda))} = \left\{ (w^*_1, w^*_2) \mid w^* \in \w(\tolp),\ w^*_1 + w^*_2 = \frac{1}{1 + \lambda} \right\}
			\end{align*}
			\item \label{remark2:partthree} The slope $m_\lambda$ of the line segment $	\mathcal{L}_{\w(\pblp^2(\lambda))} $ is given by,
			\begin{align*}
				m_\lambda = -1
			\end{align*}
			\item \label{remark2:partfour}  The end points of $\mathcal{L}_{\w(\pblp^2(\lambda))}$ are $(0, \frac{1}{1+ \lambda})$ and $\left(\frac{1}{1+ \lambda}, 0\right)$.
		\end{enumerate}
	\end{remark} 
	Remark~\propitemref{remark2}{remark2:partone} implies that if we vary the parameter value $\lambda$ in $\pblp^2$ and project all the corresponding weight sets $\w(\pblp^2(\lambda))$ in $\mathbb{R}^2$, we obtain almost the entire weight set of the associated TOLP except for the weight $(0,0,1)$ due to the reason mentioned in Theorem~\ref{thm3}. Remarks~\propitemref{remark2}{remark2:parttwo},~\propitemref{remark2}{remark2:partthree} and ~\propitemref{remark2}{remark2:partfour} are direct results of the definition of $\w(\pblp^2 (\lambda))$ for $w_1 \in \left[0,1\right]$ and are illustrated in Figure~\ref{fig3}. 
	
	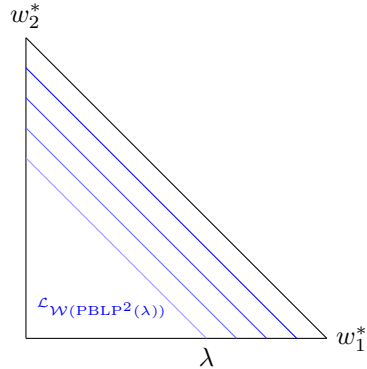
\begin{figure*}[h]
		\centering
		\begin{tikzpicture}[scale=0.8]
			\draw (0,0) -- (5,0) node[right]{$w_1^*$};
			\draw (0,0) -- (0,5) node[anchor=south]{$w_2^*$};
			\draw [] plot coordinates {(0,5) (5, 0)} node[right] [font = {\scriptsize}] {};
			\draw [blue!40] (0,3) node[left]{} -- (3,0) node[below]{};
			\draw plot[] coordinates {(0,0)} node[left] {};
			\draw []plot[] coordinates {(3, 0)} node[below] {$\lambda$};
			\draw [blue]plot[] coordinates {(2.5, 0.5)} node[left] {\tiny $	\mathcal{L}_{\w(\pblp^2(\lambda))}$};
			\draw [blue](0,4.5) node[left]{} -- (4.5,0) node[below]{};
			\draw [blue!80](0,4) node[left]{} -- (4,0) node[below]{};
			\draw [blue!60](0,3.5) node[left]{} -- (3.5,0) node[below]{};
		\end{tikzpicture}
		\caption{An illustration of $\mathcal{L}_{\w(\pblp^2(\lambda))}$ for different values of $\lambda$ in the weight set of TOLP.\label{fig3}}
	\end{figure*}
	\subsection{Illustrative examples and more results}
	
	Every solution set $S$ for TOLP contains a subset of solutions such that there is a bijection between this subset and the extreme nondominated images $Y_\en$ of TOLP.
	At the same time, the weight set $\w(\tolp)$ can be decomposed into full dimensional weight set components such that there is a bijection between these weight set components and $Y_\en(\tolp)$ (cf.~\cite{Przybylski2010}).
	By Remarks~\ref{remark1}~\ref{remark1:parttwo} and \ref{remark2}~\ref{remark2:parttwo}, for a fixed value of $\lambda\geq0$, the weight set $\w(\pblp^j(\lambda))$ is a line segment that lies in $\w(\tolp)$.
	%the line segment $\mathcal{L}_{\w(\pblp^j(\lambda))}$.
	This line segment intersects some (full-dimensional) weight set components.
	For each component $\w(y)$, there are three possibilities: 
	$\w(\pblp^j(\lambda))$ intersects $\w(y)$ either in a single vertex or along an edge, or it passes through the interior of $\w(y)$.
	The entire segment $\w(\pblp^j(\lambda))$ can then be decomposed into such intersections (cf.~\cite{Przybylski2010}).
	A solution set for $\pblp^j(\lambda)$ can be obtained by using all solutions from $S$ where the corresponding weight set component is intersected by $\w(\pblp^j(\lambda))$.
	Therefore, $S$ is also a solution set for $\pblp^j$.
	
	At the same time, Theorem~\ref{thm2} and Theorem~\ref{thm3} imply that, for every full dimensional weight set component $\w(y)$ of an extreme nondominated image $y\in Y_\en(\tolp)$, there is at least one value $\lambda\geq0$ such that the line segment $\w(\pblp^j(\lambda))$ intersects the interior of $\w(y)$.
	Then, a solution set for $\pblp^j(\lambda)$ must contain at least one solution $x$ such that $(c_1x,c_2x,d_1x)^\top=y$ (cf.~\cite{Przybylski2010}).
	Therefore, every solution set for $\pblp^j$ also contains a solution set for TOLP.
	
	Since every solution set for TOLP contains a solution set for $\pblp^j$, we can make the following statement regarding minimal solution sets:
	
	\begin{proposition}\label{prop6}
		A set $S \subseteq X$ is a minimal solution set for TOLP if and only if $S$ is a minimal solution set for $\pblp^j$.
	\end{proposition}
	
	An illustration of the weight sets of the two parametric biobjective linear programs $\pblp^1$, $\pblp^2$ and the weight set of its corresponding TOLP is shown in Example~\ref{eg1}.
	\begin{example}	\label{eg1}
		Consider the following linear $\pblp^1$ with a non-negative parameter~$\lambda \geq 0$:
		\begin{alignat*}{2}
			&\mathrm{\min} \quad &&\begin{pmatrix}
				-3x_1- x_2 + \lambda (x_1 +x_2) \\
				x_1-2 x_2
			\end{pmatrix}\\
			&\mathrm{s.\ t.} \quad &&x \in X \coloneqq \left \{x \in \mathbb{R}^2_\geq : 3x_1 +2 x_2 \geq6;	x_1 \leq 10; x_2 \leq 3\right\},
		\end{alignat*}
		and let us also consider a linear $\pblp^2$:
		\begin{alignat*}{2}
			&\mathrm{\min} \quad &&\begin{pmatrix}
				-3x_1- x_2 + \lambda (x_1 +x_2) \\
				x_1-2 x_2 +  \lambda (x_1 +x_2)
			\end{pmatrix}\\
			&\mathrm{s.\,t.} \quad &&x \in X 
		\end{alignat*}
		and the corresponding TOLP which is related to both parametric biobjective problems:
		\begin{alignat*}{2}
			&\mathrm{\min} \quad && \left(-3x_1- x_2, x_1-2 x_2, x_1 +x_2 \right)^\top\\
			&\mathrm{s.\,t.} \quad && x \in X .
		\end{alignat*}
		The weight set decomposition of the TOLP is composed of four weight set components, each corresponding to the extreme nondominated images~$y^1, y^2, y^3$, and $y^4$ as shown in Figure~\ref{fig4}.
		The weight sets of $\pblp^1(\lambda)$ and $\pblp^2(\lambda)$ for the parameter values $\lambda = 0,1,2$ are shown in Figures~\ref{fig5:a} and~\ref{fig5:b}, respectively.
	\end{example}
	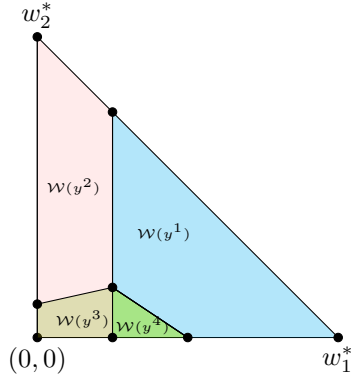
\begin{figure*}[h]
		\centering
		\begin{tikzpicture}[scale=0.8]
			\draw (0,0) -- (5,0) node[below]{$w_1^*$};
			\draw (0,0) -- (0, 5) node[anchor=south]{$w_2^*$};
			\draw plot[mark=*] coordinates {(0,0)} node[below]{$(0,0)$} ;
			\draw [fill=pink, fill opacity = 0.3] plot coordinates {(0, 5/9) (0, 5) (1.25, 3.75) (1.25, 2.5/3)};
			\draw [fill=olive, fill opacity = 0.3] plot coordinates {(0,0) (0, 5/9) (1.25, 2.5/3)(2.5,0)};
			\draw [fill=green, fill opacity = 0.25] plot coordinates {(1.25, 0) (1.25, 2.5/3) (2.5,0)};
			\draw [fill=cyan, fill opacity = 0.3] plot coordinates {(2.5,0) (1.25, 2.5/3) (1.25, 3.75)(5, 0)};
			\draw plot[mark=*] coordinates {(5,0)};
			\draw plot[mark=*] coordinates {(1.25, 3.75)};
			\draw plot[mark=*] coordinates {(0, 5/9)};
			\draw plot[mark=*] coordinates {(1.25,0)};
			\draw plot[mark=*] coordinates {(2.5,0)};
			\draw plot[mark=*] coordinates {(0,5)};
			\draw plot[mark=*] coordinates {(1.25, 2.5/3)};
			\draw (1.25,0)--(1.25,5/6);
			\draw (1.25, 3.75)--(1.25,5/6);
			\draw (2.5,0)--(1.25,5/6);
			\node at (.6,2.5) {\tiny $\w(y^2)$};
			\node at (0.75,0.3) {\tiny $\w(y^3)$};
			\node at (2.1,1.75) {\tiny $\w(y^1)$};
			\node at (1.75,.2) {\tiny $\w(y^4)$};
		\end{tikzpicture}
		\caption{Weight set of TOLP with four weight set components.\label{fig4}}
	\end{figure*}
	
	\begin{figure*}[h]
		\begin{subfigure}[b]{0.49\textwidth}
			\centering
			\begin{tikzpicture}[scale=0.4]
				\draw (0,0) -- (10,0) node[below]{$w_1^*$};
				\draw (0,0) -- (0,10) node[anchor=south]{$w_2^*$};
				%\draw plot coordinates {(0,10) (10, 0)};
				\draw [red] plot coordinates{(0,10) (10, 0)};
				\draw [blue] plot coordinates {(0,10) (5, 0)};
				\draw [orange] plot coordinates {(0,10) (10/3, 0)};
				\draw [thick, red]plot[] coordinates {(8.5,5)} node[left] {\tiny $\mathcal{L}_{\w(\pblp^1 (0)}$};
				\draw [thick, blue]plot[] coordinates {(7.2,1.8)} node[left] {\tiny $\mathcal{L}_{\w(\pblp^1 (1)}$};
				\draw [thick, orange]plot[] coordinates {(2.5,3)} node[left] {\tiny $\mathcal{L}_{\w(\pblp^1 (2)}$};
				\draw [fill=pink, fill opacity = 0.25,ultra thin, dash pattern=on 1pt off 2pt] plot coordinates {(0, 10) (0, 10/9) (2.5, 5/3) (2.5, 7.5)};
				\draw [fill=olive, fill opacity = 0.25, ultra thin, dash pattern=on 1pt off 2pt] plot coordinates {(0,0) (0, 10/9) (2.5, 5/3)(2.5,0)};
				\draw [fill=green, fill opacity = 0.25,ultra thin, dash pattern=on 1pt off 2pt] plot coordinates {(2.5, 0) (2.5, 5/3) (5,0)};
				\draw [fill=cyan, fill opacity = 0.25, ultra thin, dash pattern=on 1pt off 2pt] plot coordinates { (2.5, 7.5) (2.5, 5/3) (5,0) (10, 0)};
				\draw plot[mark=*] coordinates {(0,0)} node[below]{} ;
				\draw plot[mark=*] coordinates {(10,0)};
				\draw plot[mark=*] coordinates {(5/2, 15/2)};
				\draw plot[mark=*] coordinates {(5/2,5)};
				\draw plot[mark=*] coordinates {(5,0)};
				\draw plot[mark=*] coordinates {(5/2,5/2)};
				\draw plot[mark=*] coordinates {(20/7, 10/7)};
				\draw plot[mark=*] coordinates {(10/3, 0)};
			\end{tikzpicture}
			\caption{\label{fig5:a}}
		\end{subfigure}
		\begin{subfigure}[b]{0.49\textwidth}
			\centering
			\begin{tikzpicture}[scale=0.4]
				\draw (0,0) -- (10,0) node[below]{$w_1^*$};
				\draw (0,0) -- (0,10) node[anchor=south]{$w_2^*$};
				%\draw plot coordinates {(0,10) (10, 0)};
				\draw [red] plot coordinates{(0,10) (10, 0)};
				\draw [ blue] plot coordinates {(0,5) (5, 0)};
				\draw [orange] plot coordinates {(0,10/3) (10/3, 0)};
				\draw [thick, red]plot[] coordinates {(8.5,5)} node[left] {\tiny $\mathcal{L}_{\w(\pblp^2 (0)}$};
				\draw [thick, blue]plot[] coordinates {(6.5,2)} node[left] {\tiny $\mathcal{L}_{\w(\pblp^2 (1)}$};
				\draw [thick, orange]plot[] coordinates {(0,10/3)} node[left] {\tiny $\mathcal{L}_{\w(\pblp^2 (2)}$};
				\draw [fill=pink, fill opacity = 0.25,ultra thin, dash pattern=on 1pt off 2pt] plot coordinates {(0, 10) (0, 10/9) (2.5, 5/3) (2.5, 7.5)};
				\draw [fill=olive, fill opacity = 0.25,ultra thin, dash pattern=on 1pt off 2pt] plot coordinates {(0,0) (0, 10/9) (2.5, 5/3)(2.5,0)};
				\draw [dashed, fill=green, fill opacity = 0.25] plot coordinates {(2.5, 0) (2.5, 5/3) (5,0)};
				\draw [fill=cyan, fill opacity = 0.25,ultra thin, dash pattern=on 1pt off 2pt] plot coordinates { (2.5, 7.5) (2.5, 5/3) (5,0) (10, 0)};
				\draw plot[mark=*] coordinates {(0,0)} node[below]{} ;
				\draw plot[mark=*] coordinates {(10,0)};
				\draw plot[mark=*] coordinates {(5/2, 15/2)};
				\draw plot[mark=*] coordinates {(5/2,5/6)};
				\draw plot[mark=*] coordinates {(5,0)};
				\draw plot[mark=*] coordinates {(5/2,5/2)};
				\draw plot[mark=*] coordinates {(10/3, 0)};
			\end{tikzpicture}
			\caption{\label{fig5:b}}
		\end{subfigure}
		\caption{Illustration of line segments of $\pblp^j $ for parameter values, $\lambda=0,1,2$ \label{fig5} in the weight set of TOLP. (\ref{fig5:a}) Line segments $\mathcal{L}_{\w(\pblp^1 (\lambda)}$ of $\pblp^1$ and (\ref{fig5:b}) Line segments $\mathcal{L}_{\w(\pblp^1 (\lambda)}$ of $\pblp^2$}.
	\end{figure*}
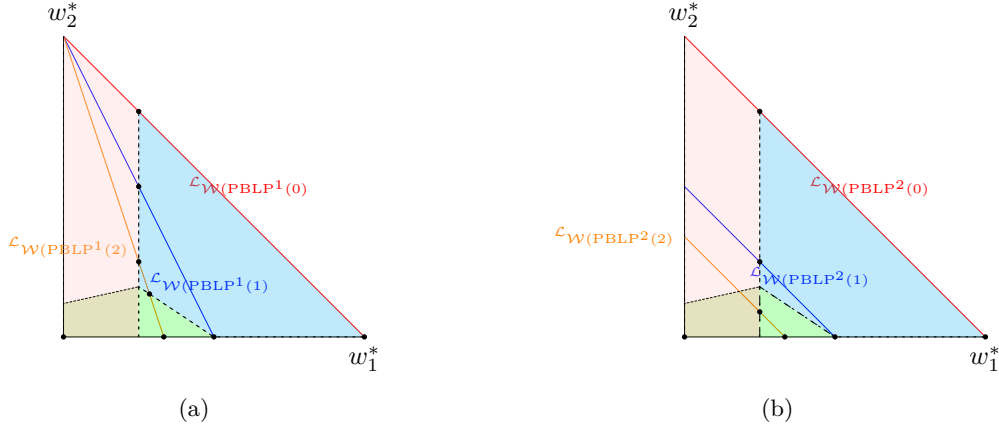
	
	%Since solving a parametric biobjective linear program is understood as finding a minimal solution set for all values of $\lambda$, it is sufficient to focus on the parameter values where a minimal solution set of BOLP $\pblp^j(\lambda)$ changes as $\lambda$ varies. These parameter values are known as \emph{breakpoints}. 
	
	A minimal solution set for the parametric biobjective linear program $\pblp^j$ is defined over the entire parameter set. We are particularly interested in the critical parameter values where a minimal solution set of $\pblp^j(\lambda)$ changes as $\lambda$ varies. These values are known as \emph{breakpoints}.
	More formally, for a given $\pblp^j$, some parameter value $\lambda \geq 0$ is called a breakpoint if, for a small $\varepsilon > 0$, 
	\begin{align*}
		S(\pblp^j(\lambda + \varepsilon)) \neq S(\pblp^j(\lambda - \varepsilon)).
	\end{align*}
	% parametric biobjective linear programs $\pblp^j$ is a set of breakpoints such that each breakpoint corresponds to a unique set of solutions that is a subset of the minimal solution set, $S$. 
	Therefore, the solution to parametric biobjective linear programs $\pblp^j$ consists of the following:
	\begin{enumerate}[label=\roman*.]
	\item a minimal solution set $S$ for $\pblp^j$ and 
	\item a set of breakpoints in the parameter set. 
	\end{enumerate}
	
	In $\pblp^j$, breakpoints can exhibit behaviour not observed in single-objective parametric optimization. 
	Certain breakpoints are characterized by a unique minimal solution set that changes at that point, differing from the sets immediately preceding and following it.
	This happens when the solutions leaving and entering a minimal solution set correspond to extreme nondominated images that are still nondominated but not extreme points for $\pblp^j(\lambda)$. 
	As a result, the parameter set is divided into a set of intervals and/or unique parameter values, each corresponding to a minimal solution set, respectively. 
	A special case here is that of a \emph{tie}; we define a \emph{tie} as a breakpoint that serves as both the upper bound of an interval and the lower bound of the next consecutive interval. 
	In such cases, we include the breakpoint in the preceding interval and exclude it from the following one. 
	%We show these different types of breakpoints in Figure~\ref{fig6} and  in our illustrations, note that feasible solutions $x_1, x_2, x_3$ and $x_4$ correspond to the extreme nondominated images $y_1,y_2,y_3$ and $y_4$, respectively.
	These various breakpoints are visualized in Figure~\ref{fig6}.
	
	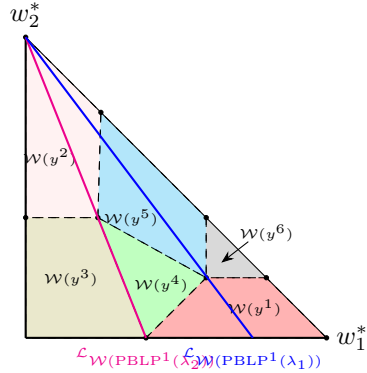
\begin{figure*}[h]
		\centering
		\begin{tikzpicture}[scale=0.4]
			\draw [dashed,fill=pink, fill opacity = 0.2] plot coordinates {(0,10) (0,4)(2.4,4) (2.5,7.5) (0,10)};
			\draw [dashed,fill=olive, fill opacity = 0.2] plot coordinates {(0,0) (0,4) (2.4,4)(4,0)(0,0)};
			\draw [dashed,fill=green, fill opacity = 0.25] plot coordinates { (2.4,4)(4,0)(6,2) (2.4,4)};
			\draw [dashed,fill=cyan, fill opacity = 0.3] plot coordinates {(2.4,4) (2.5,7.5)(6,4) (6,2) (2.4, 4)};
			\draw [dashed,fill=gray, fill opacity = 0.3] plot coordinates {(8,2)(6,2)(6, 4)};
			\draw [dashed,fill=red, fill opacity = 0.3] plot coordinates {(4,0) (6,2)(8,2) (10,0) (4,0)};
			\draw [-Stealth] (7.4,3.2) -- (6.5,2.5) ;
			\draw [thick](0,0) -- (10,0) node[right]{$w_1^*$};
			\draw [thick](0,0) -- (0, 10) node[anchor=south]{$w_2^*$};
			\draw (0,10) -- (10,0);
			\draw plot[mark=*] coordinates {(0,10)};
			\draw plot[mark=*] coordinates {(0,4)};
			\draw plot[mark=*] coordinates {(2.5,7.5)};
			\draw plot[mark=*] coordinates {(2.4,4)};
			\draw plot[mark=*] coordinates {(4,0)};
			\draw plot[mark=*] coordinates {(10,0)};
			\draw plot[mark=*] coordinates {(6,2)};
			\draw plot[mark=*] coordinates {(8,2)};
			\draw plot[mark=*] coordinates {(6,4)};
			\node at (7.6,1) {\tiny $\w(y^1)$};
			\node at (0.8,6) {\tiny $\w(y^2)$};
			\node at (1.5,2) {\tiny $\w(y^3)$};
			\node at (4.5,1.8) {\tiny $\w(y^4)$};
			\node at (3.5,4) {\tiny $\w(y^5)$};
			\node at (8,3.5) {\tiny $\w(y^6)$};
			\draw [thick, magenta] (0,10)--(4,0)node[below]{\tiny $\mathcal{L}_{\w(\pblp^1(\lambda_2))}$};
			\draw [thick, blue](0,10)--(7.55,0)node[below]{\tiny $\mathcal{L}_{\w(\pblp^1(\lambda_1))}$};
		\end{tikzpicture}
		\caption{A solution set at the breakpoint $\lambda_1$, $S(\pblp^1(\lambda_1 )) =\left\{x^1, x^5, x^2\right\}$ is different from $S(\pblp^1(\lambda_1 - \varepsilon)) = \left\{x^1, x^6, x^5, x^2\right\}$ and $S(\pblp^1(\lambda_1 + \varepsilon))=\left\{x^1, x^4, x^5, x^2\right\}$, while $\w(\pblp^1(\lambda_2))$ shows the case of a tie where $S(\pblp^1(\lambda_2)) =\left\{x^4, x^2\right\} = \left\{x^3, x^2\right\}$. Note that feasible solutions such as $x_1, x_2, x_3$, and $x_4$ map to the extreme nondominated images, $y_1,y_2,y_3$, and $y_4$, respectively. \label{fig6}}
	\end{figure*}

	\begin{proposition}\label{prop7}
		If $\lambda_i$ is a breakpoint of a $\pblp^j$ with a corresponding weight set $\w(\pblp^j (\lambda_i))$,  there exists at least one extreme weight in $\w(\tolp)$ on the intersection of $\mathcal{L}_{\w(\pblp^j (\lambda_i))}$ and $\w(\tolp)$. 
	\end{proposition}
	\begin{proof}
		The proof uses the following definition. Let $\lambda \geq 0$ be the parameter, $Y_{\en}(\tolp)$ denote the set of extreme nondominated images of TOLP, and $\w(\pblp^j(\lambda))$ be the weight set of $\pblp^j(\lambda)$. We define
		\begin{align*}
			\nu(\lambda) \coloneqq \left \{\w(y): y \in Y_\en(\tolp), \interior (\w(y)) \cap \w(\pblp^j (\lambda_i)) \neq \emptyset, j =1, 2 \right\}.
		\end{align*}
		
		Let $\lambda_i$ and $\lambda_{i+1}$ be two consecutive breakpoints. We consider both cases separately.
		
		\textbf{Case 1: $j = 1$}
		By Proposition~\ref{prop3}, it is trivial that the extreme weight $(0,1,0) \in \w(\tolp)$ is contained in $\w(\pblp^1(\lambda))$ for all $\lambda \geq 0$, whether $\lambda$ is a breakpoint or not. 
		%To ensure unique weight sets for $\w(\pblp^1(\lambda))$, we exclude $(0,1,0)$ for this case. 
		However, this proposition still holds for $\pblp^1$ when we exclude $(0,1,0)$ from the weight sets of $\w(\pblp^1(\lambda))$.
		
		By definition of a breakpoint, $S(\pblp^1(\lambda_i)) \neq S(\pblp^1(\lambda_{i+\varepsilon}))$.
		Suppose, for contradiction, that no extreme weight of $\w(\tolp)$ lies on the intersection of $\mathcal{L}_{\w(\pblp^1(\lambda_i))} \setminus \{(0,1)\}$ and $\w(\tolp)$. Then the line segment $\mathcal{L}_{\w(\pblp^1(\lambda_i))} \setminus \{(0,1)\}$ passes entirely through the relative interiors of weight set components in $\nu(\lambda_i)$. By the definition of relative interior and the compactness of $\nu(\lambda_i)$, there exists $\delta > 0$ such that
		\begin{align*}
			\nu (\lambda_i +\delta)= \nu (\lambda_i -\delta) = \nu(\lambda_{i}).
		\end{align*}
		This implies that the line segments $\mathcal{L}_{\w(\pblp^2(\lambda))}$ for $\lambda = \left\{i-\delta, i, i+\delta\right\}$ passes through same set of weight set components. Therefore, 
		\begin{align*}
			S (\pblp^1 (\lambda_i +\delta))= S (\pblp^1 (\lambda_i -\delta)) = S (\pblp^1(\lambda_i)).
		\end{align*}
		This contradicts $\lambda_i$ being a breakpoint.
		
		\textbf{Case 2: $j = 2$}
		By Proposition~\ref{prop5}, all weight sets $\w(\pblp^2(\lambda))$ are unique for $\lambda \geq 0$. The definition of a breakpoint gives $S(\pblp^2(\lambda_i)) \neq S(\pblp^2(\lambda_{i+1}))$.
		
		Suppose, for contradiction, that no extreme weight of $\w(\tolp)$ lies on the intersection of $\mathcal{L}_{\w(\pblp^2(\lambda_i))}$ and $\w(\tolp)$. Then $\mathcal{L}_{\w(\pblp^2(\lambda_i))}$ passes through relative interiors of weight set components in $\nu(\lambda_i)$. By the relative interior properties and compactness of $\nu(\lambda_i)$, there exists $\delta > 0$ such that
		\begin{align*}
			\nu (\lambda_i +\delta)= \nu (\lambda_i -\delta) = \nu(\lambda_{i}).
		\end{align*}
		This implies that the line segments $\mathcal{L}_{\w(\pblp^2(\lambda))}$ for $\lambda = \left\{i-\delta, i, i+\delta\right\}$ passes through the same set of weight set components. Therefore,  
		\begin{align*}
			S (\pblp^2 (\lambda_i +\delta))= S (\pblp^2 (\lambda_i -\delta)) = S (\pblp^2(\lambda_i)).
		\end{align*}
		This contradicts $\lambda_i$ being a breakpoint.
	\end{proof}
	
	Note that not every extreme weight leads to a unique breakpoint.
	For a given breakpoint $\lambda_i$, there might be several extreme weights in the line segment $\mathcal{L}_{\w(\pblp^j (\lambda_i))}$.
	We demonstrate this with Example~\ref{eg2}.
	\begin{example}	\label{eg2}
		Consider the following linear triobjective linear program,
		\begin{alignat*}{3}
			&\mathrm{min} \quad && (x_1, x_2, x_3)^\top \\
			&\mathrm{s.t.} \quad && x \in X\coloneqq \{ x \in \mathbb{R}^3_\geqq \!: \quad && 2x_1 + 3x_2 + 5x_3  \geq 40, \\
			& \quad && && 2x_1 + 15x_2 - 15x_3 \geq 0, \\
			& \quad && && 2x_1 - x_2 + x_3 \geq 0, \\
			& \quad && && 2x_1 - x_2 - 15x_3 \leq 0 \}
		\end{alignat*}
	\end{example}
	
	The problem has the following nondominated image set which is a convex hull of three extreme nondominated images, $y^1 =(5,10,0)$, $y^2=(0,5,5)$, and $y^3 = (15,0,2)$ as shown in Figure~\ref{fig7}. 
	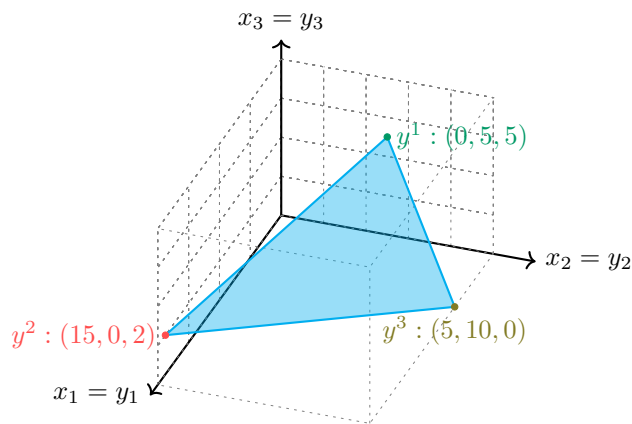
\begin{figure*}[h]
		\centering
		\tdplotsetmaincoords{60}{110}
		\begin{tikzpicture}[scale=0.3,tdplot_main_coords]
			\tikzset{
				cube/.style={very thick,black},
				grid/.style={ultra thin, gray,  dash pattern=on 1pt off 2pt},  % DEFINE grid style here
				axis/.style={->,blue,thick}
			}
			% Draw the grid first (in the background)
			\foreach \y in {0,2,4,6,8,10}
			\foreach \z in {0,2,4,6,8}
			{
				\draw[grid] (0,\y,\z) -- (0,\y,8);    % Vertical grid lines (constant y)
				\draw[grid] (0,0,\z) -- (0,10,\z);    % Horizontal grid lines (constant z)
			}
			% Draw the grid first (in the background)
			\foreach \x in {0,4,8,12,16}
			\foreach \z in {0,2,4,6,8}
			{
				\draw[grid] (\x,0,0) -- (\x,0,8);    % Vertical lines along x-axis
				\draw[grid] (0,0,\z) -- (16,0,\z);    % Horizontal lines along z-axis
			}
			
			% Draw axes next
			\draw [thick, ->](0,0,0) -- (17,0,0) node [left] {$x_1=y_1$};
			\draw [thick, ->](0,0,0) -- (0,12,0) node [right] {$x_2=y_2$};
			\draw [thick, ->](0,0,0) -- (0,0,9) node [above] {$x_3=y_3$};
			
			% Draw the cube outlines with thinner dashes
			\draw[gray, ultra thin, dash pattern=on 1pt off 2pt] (0,0,0) -- (0,10,0) -- (16,10,0) -- (16,0,0) -- cycle;
			\draw[gray, ultra thin, dash pattern=on 1pt off 2pt] (0,0,8) -- (0,10,8) -- (16,10,8) -- (16,0,8) -- cycle;
			\draw[gray, ultra thin, dash pattern=on 1pt off 2pt] (16,0,0) -- (16,0,8);
			\draw[gray, ultra thin, dash pattern=on 1pt off 2pt] (0,10,0) -- (0,10,8);
			\draw[gray, ultra thin, dash pattern=on 1pt off 2pt] (16,10,0) -- (16,10,8);
			
			% Draw the blue polytope LAST so it appears on top
			\fill[cyan, opacity=0.4] (5,10,0) -- (0,5,5) -- (15,0,2) -- cycle;
			\draw [cyan, thick] (5,10,0) -- (0,5,5) -- (15,0,2) -- cycle;
			
			% Draw the points last so they appear on top of everything
			\draw [olive!80!black]plot[mark=*,mark size=4pt] coordinates {(5,10,0)}node[below] {$y^3 :(5,10,0)$};
			\draw [green!60!blue]plot[mark=*,mark size=4pt] coordinates {(0,5,5)} node[right] {$y^1: (0,5,5)$};
			\draw [red!70] plot[mark=*,mark size=4pt] coordinates {(15,0,2)} node[left] {$y^2 : (15,0,2)$};
		\end{tikzpicture}
		\caption{An illustration of the nondominated image set of the triobjective problem.\label{fig7}}
	\end{figure*}
	
	The weight set of TOLP, Example~\ref{eg2} and the line segment representing the weight set of the corresponding $\pblp^2$, where $\lambda = 1$ are shown below. 
	\definecolor{emerald}{rgb}{0.21, 0.78, 0.9}
	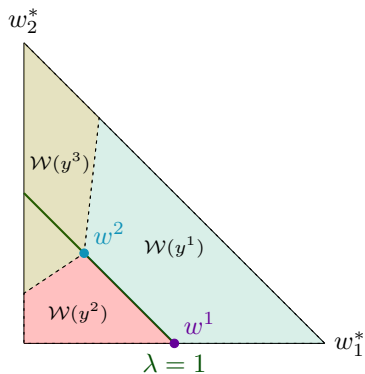
\begin{figure*}[h]
		\centering
		\begin{tikzpicture}[scale=0.8]
			\draw (0,0) -- (5,0) node[right]{$w_1^*$};
			\draw (0,0) -- (0,5) node[anchor=south]{$w_2^*$};
			\draw plot coordinates {(0,5) (5, 0)};
			\draw [thick, green!30!black] plot coordinates {(0,2.5) (2.5, 0)} node[below] {$\lambda = 1$};
			\draw [dash pattern=on 1pt off 2pt,fill=red, fill opacity = 0.25] plot coordinates {(1, 3/2) (0, 5/6) (0,0) (2.5,0)};
			\draw [dash pattern=on 1pt off 2pt, fill=green!60!blue, fill opacity = 0.15] plot coordinates { (1, 3/2) (5/4,15/4) (5, 0) (2.5, 0)};
			\draw [dash pattern=on 1pt off 2pt, fill=olive, fill opacity = 0.25] plot coordinates {(0, 5/6) (1, 3/2) (5/4, 15/4) (0,5)};
			%\draw plot[mark=*] coordinates {(2.5, 7.5)};
			\node at (.6,3) {\scriptsize $\w(y^3)$};
			\node at (0.9,0.5) {\scriptsize $\w(y^2)$};
			\node at (2.5,1.6) {\scriptsize $\w(y^1)$};
			\draw [cyan!70!black] plot[mark=*] coordinates {(1, 3/2)}node[above right] {$w^2$};
			\draw [red!40!blue] plot[mark=*] coordinates {(2.5, 0)} node[above right] {$w^1$};
		\end{tikzpicture}
		\caption{The line segment $\mathcal{L}_{\w(\pblp^2(1))}$ of $\pblp^2(\lambda)$ at $\lambda =1$ intersects with $\w(\tolp)$ at two extreme weights $w^1$ and $w^2$. \label{fig8}}
	\end{figure*}
	We observe from Figure~\ref{fig8} that we have two different extreme weights $w^1 =\left(\frac{1}{2},0,\frac{1}{2} \right)$ and $w^2 = \left(\frac{1}{5},\frac{3}{10},\frac{1}{2}\right)$ which correspond to the same breakpoint $\lambda_1 = 1$. This is also due to the special case of a tie in the solution sets of $\pblp^2(\lambda)$ before and after the breakpoint, which was discussed earlier.\\
	Regarding the nondominated image set of $\pblp^j$, this scenario implies that the extreme nondominated images~$y_{\lambda}^2$ and $y_{\lambda}^1$ asymptotically converge to a common nondominated image at $\lambda =1$.
	We illustrate this by using the image set of the corresponding non-parametric BOLP $\pblp^2(\lambda)$ with varying $\lambda$, $\lambda \coloneqq \left \{0, 1, 2\right\}$. Consider $\pblp^2$ of the same Example~\ref{eg2} i.\,e., 
	\begin{alignat*}{2}
		& \mathrm{\min} \quad &&
		\begin{pmatrix}
			x_1 + \lambda x_3\\
			x_2 + \lambda x_3
		\end{pmatrix}\\
		& \mathrm{s.\ t.} \quad && x \in X.
	\end{alignat*}
	It has an image set with the axes $y_1 = x_1 + \lambda x_3$ and $y_2 = x_2 + \lambda x_3$ that changes continuously with the variation of parameter $\lambda$ as shown in Figure~\ref{fig9}. Note that the two extreme nondominated images~$y_{\lambda}^2$ and $y_{\lambda}^1$ converge to the same point, i.e., $y_{1}^2 = y_{1}^1$ at $\lambda = 1$.

	\begin{figure*}[h]
		\begin{subfigure}[b]{0.3\textwidth}
			\centering
			\begin{tikzpicture}[scale=0.16]
				\draw[->] (0,0) -- (20,0) node[below] {\tiny $y_1 = x_1 + \lambda x_3$};
				\draw[->] (0,0) -- (0,20) node[above]{\tiny $y_2 = x_2 + \lambda x_3$};
				\draw [draw=none,fill= gray, fill opacity = 0.35] plot coordinates {(0, 5)(5, 10) (15,0)};
				\draw [thick, red] plot coordinates {(0, 5) (15,0)};
				\draw [red!50!blue]plot[mark=*,mark size=10pt] coordinates {(0, 5)} node[below right] {\scriptsize $y_{0}^1$};
				\draw [blue]plot[mark=*,mark size=10pt] coordinates {(5, 10)} node[left] {\scriptsize $y_{0}^3$};
				\draw [emerald]plot[mark=*,mark size=10pt] coordinates {(15,0)} node[above] {\scriptsize $y_{0}^2$};
				\draw[-Stealth] [red!50!blue](0,5) --(2, 7) node[right] {};
				\draw[-Stealth] [emerald](15,0) --(17,2) node[right] {};
			\end{tikzpicture}
			\caption{$\lambda = 0$}
			\label{fig8a}
		\end{subfigure}
		\hfill
		\begin{subfigure}[b]{0.3\textwidth}
			\centering
			\begin{tikzpicture}[scale=0.16]
				\draw[->] (0,0) -- (20,0) node[below] {\tiny$y_1 = x_1 + \lambda x_3$};
				\draw[->] (0,0) -- (0,20) node[above]{\tiny $y_2 = x_2 + \lambda x_3$};
				\draw [thick, red](5, 10)--(17, 2);
				\draw [red!50!blue]plot[mark=*,mark size=10pt] coordinates {(5, 10)} node[above] {\scriptsize $y_{1}^1$};
				\draw [blue]plot[mark=*,mark size=10pt] coordinates {(5, 10)} node[left] {\scriptsize $y_{1}^3$};
				\draw [emerald]plot[mark=*,mark size=10pt] coordinates {(17,2)} node[above] {\scriptsize $y_{1}^2$};
				\draw[-Stealth] [red!50!blue](5,10) --(7,12) node[right] {};
				\draw[-Stealth] [emerald](17,2) --(19,4) node[right] {};
			\end{tikzpicture}
			\caption{$\lambda = 1$}
			\label{fig8b}
		\end{subfigure}
		\hfill
		\begin{subfigure}[b]{0.3\textwidth}
			\centering
			\begin{tikzpicture}[scale=0.16]
				\draw[->] (0,0) -- (24,0) node[below] {\tiny $y_1 = x_1 + \lambda x_3$};
				\draw[->] (0,0) -- (0,20) node[above]{\tiny $y_2 = x_2 + \lambda x_3$};
				\draw [draw=none,fill=lightgray, fill opacity = 0.3] plot coordinates {(5,10) (10,15)(19,4)};
				\draw [thick, red] plot coordinates {(5,10) (19,4)};
				\draw [red!50!blue]plot[mark=*,mark size=10pt] coordinates {(10, 15)} node[above] { \scriptsize $y_{2}^1$};
				\draw [blue]plot[mark=*,mark size=10pt] coordinates {(5, 10)} node[left] {\scriptsize $y_{2}^3$};
				\draw [emerald]plot[mark=*,mark size=10pt] coordinates {(19,4)} node[below] {\scriptsize $y_{2}^2$};
				\draw[-Stealth] [red!50!blue](10,15) --(12,17) node[right] {};
				\draw[-Stealth] [ emerald](19,4) --(21,6) node[right] {};
			\end{tikzpicture}
			\caption{$\lambda = 2$}
			\label{fig8c}
		\end{subfigure}
		\caption{An illustration of change in the nondominated set for $\pblp^2(\lambda)$. (\ref{fig8a}): The images $y_{0}^2$ and $y_{0}^3$ are the extreme nondominated images for $\pblp^2(0)$. (\ref{fig8b}): The images $y_{1}^2$ and $y_{1}^1$ are equivalent and $Y_\en (\pblp^2(1)) = \left\{y_{1}^2, y_{1}^3\right\} =\left\{y_{1}^1, y_{1}^3 \right\}$. (\ref{fig8c}): The image $y_{2}^2$ is no longer an extreme nondominated image for $\pblp^2(2)$ and $Y_\en (\pblp^2(2)) = \left\{y_{2}^1, y_{2}^3\right\}$. \label{fig9}} 
	\end{figure*}
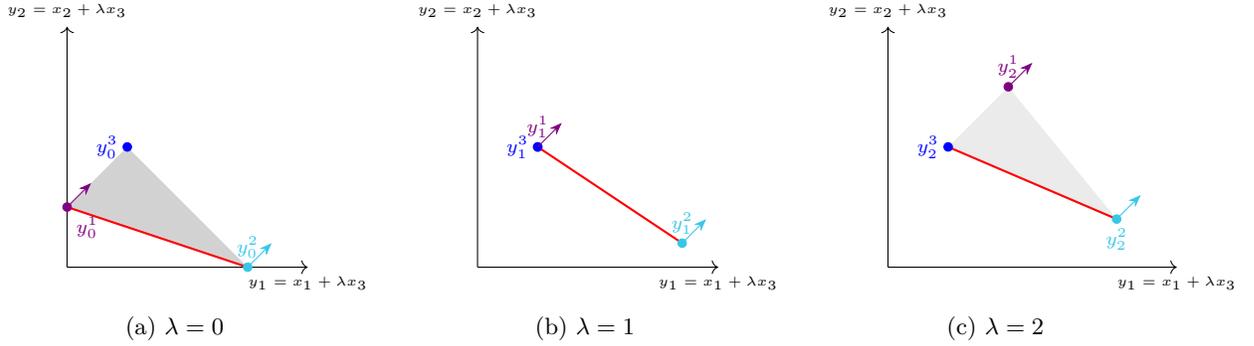
	
	\begin{corollary}\label{coro1}
		The number of breakpoints is finite. 
	\end{corollary}
	\begin{proof}
		From Proposition~\ref{prop7}, we know that the number of breakpoints is bounded by the number of extreme weights in the weight set of TOLP. Moreover, the number of extreme weights in $\w(\tolp)$ is finite due to the finiteness of the extreme points of all weight set components. Thus, the number of breakpoints is finite. 
	\end{proof}
	
	\section{Algorithms}\label{sec4}
	
	As shown in Proposition~\ref{prop7}, a minimal solution set of $\pblp^j$ can be obtained by computing a minimal solution set of $\tolp$.
	Thus, established algorithms from multi-objective optimization can be used directly to compute minimal solution sets for $\pblp^j$.
	However, when solving $\pblp^j$, we are also interested in the set of breakpoints.
	In this section, we propose two approaches for computing the breakpoints of $\pblp^j$.
	The first approach uses a minimal solution set~$S$ that is computed by any existing algorithm for $\tolp$.
	For each solution $x$ in $S$, two linear programs are solved to determine the parameter interval consisting of all values of $\lambda$ for which $x$ is in a minimal solution set of $\pblp^j(\lambda)$.
	The union of the interval boundaries of all solutions in $S$ is the set of breakpoints.
	In contrast, the second approach requires the usage of an algorithm for $\tolp$ that computes its weight set decomposition by design (for example,~\cite{Przybylski2010}).
	Then, the aforementioned intervals can be computed without additional overhead.
	
	\subsection{Breakpoint enumeration algorithm}
	
	%Although we have shown that the minimal solution set of the parametric biobjective linear programs and the corresponding triobjective linear program is the same, the goal of solving $\pblp$ is to compute the set of breakpoints. 
	In this algorithm, we want to find the set of breakpoints of~$\pblp^1$ and~$\pblp^2$ using the weight set components of all $y \in Y_\en$ from the corresponding triobjective linear program. 
	The basic idea of the algorithm is to first use some existing algorithm to compute $Y_\en(\tolp)$. 
	Then, for every extreme nondominated image $y$, we look at its weight set component and compute its boundaries to find the parameter intervals with respect to an optimal solution of $\pblp^j$.
	
	Every weight set component~$\w(y)$ corresponds to a solution $x \in S$ such that $x$ is optimal for $\pblp^j(\lambda)$ for some $\lambda \geq 0$ (see the proof of Proposition~\ref{prop6}).
	In Remarks~\propitemref{remark1}{remark1:parttwo}~and~\propitemref{remark2}{remark2:parttwo}, we have established that for every parameter value $\lambda \geq 0$, the weight sets of $\pblp^j(\lambda)$ are line segments intersecting $\w(\tolp)$. 
	Moreover, each weight set component of $\w(\tolp)$ is intersected by at least one line segment. 
	In particular, every weight set component is bounded by two specific line segments that mark its start and end. 
	
	For each weight set component $\w(y)$, we are interested in finding these two bounding line segments and their corresponding parameter values $\lambda_\ell$ and $\lambda_u$.
	The two bounding line segments define the interval $\left[\lambda_\ell, \lambda_u\right]$ where $\lambda_\ell$ is the minimum and $\lambda_u$ is the maximum parameter value, as shown in Figure~\ref{fig10}. 
	This parameter interval consists of the parameter values for which the corresponding solution $x \in S$ of $y$ is optimal for $\pblp^j(\lambda)$ where $\lambda \in [\lambda_\ell, \lambda_u]$. 
	
	More precisely, for each weight set component $\w(y)$ of $\w(\tolp)$, we solve the following program for $\pblp^1$:
	\begin{alignat*}{3} \label{eq:minimumlambda1}
		& \min \quad && \lambda \\
		& \text{s.t.} \quad && w \in \w(y), \\
		&  && w \in \w(\pblp^1(\lambda)), && \tag{$P^1_\wsc(\lambda)$} \\
		&  && w \in \mathbb{R}_{\geqq}^3 \setminus \{(0,1,0)\}
	\end{alignat*}
	to determine the lower parameter bound $\lambda_\ell$, and solve the corresponding maximization program to find $\lambda_u$.
	%And if we use $(0,1,0)$ as an extreme point of any weight set component then the resulting parameter interval is $[0,\infty)$. 
	We exclude $(0,1,0)$ as an extreme point of any weight set component because it results in a parameter interval of $[0,\infty)$ for some solutions, which can be misleading as the actual interval is bounded. 
	Instead, for these weight set components, we use other extreme weights to determine the correct parameter interval.
	%Additionally, it serves as the shared point amongst all the line segments $\mathcal{L}_{\w(\pblp^1(\lambda))}$ for all $\lambda \geq 0$ (see case $j =1$ in the proof of Proposition~\ref{prop7}). 

	For $\pblp^2$, we solve the analogous program for each weight set component $\w(y)$:
	\begin{alignat*}{3} \label{eq:minimumlambda2}
		& \min \quad && \lambda \\
		& \text{s.t.} \quad && w \in \w(y), \\
		&  && w \in \w(\pblp^2(\lambda)), && \tag{$P^2_\wsc(\lambda)$} \\
		&  && w \in \mathbb{R}_{\geqq}^3.
	\end{alignat*}
	
	We formulate the weight set component $\w(y)$ as constraints in $P^j_\wsc(\lambda)$ (cf.~\cite{Benson1998}):
	\begin{equation} \label{eq:weightsetcomponent}
		\w(y) = \left\{ (w_1, w_2, w_3): \quad
		\begin{aligned}
			& \begin{alignedat}{2}
				A^\top v & - & C^\top w & \geqq 0 \\
				b^\top v & - & y^\top w & = 0 \\
				& & \mathbf{1}^\top w & = 1 \\
			\end{alignedat}\\
			& v \in \mathbb{R}^m_\geqq, w \in \mathbb{R}_{\geqq}^3
		\end{aligned}
		\right\}
	\end{equation}
	where $C\in \mathbb{Q}^ {3 \times n}$ consists of rows $c_1, c_2$ and $d_1$.
	For notational simplicity in describing the algorithm, we temporarily define a weight set component $\mathcal{W}(y)$ as 
	\begin{align}\label{eqn:10}
		\w(y)  \coloneqq \left\{(w_1,w_2, w_3): Pw \geqq q, w \in \mathbb{R}^{m+3}_\geqq\right\}
	\end{align}
	where $P \in \mathbb{Q}^{n+4\times m+3}$ represents a coefficient matrix and $q \in \mathbb{Q}^{n+4} $ is a right-hand side vector. 
	
	We begin with the problem $\pblp^2$ because it is easier and will be followed by its adaptation to $\pblp^1$. 
	For $\pblp^2$, we substitute the description of the line segment, $\mathcal{L}_{\w(\pblp^2(\lambda))}$ from Remark~\propitemref{remark2}{remark2:parttwo} for $\w(\pblp^2(\lambda))$ and use the definition of $\w(y)$ from (\ref{eqn:10}) into the program~\ref{eq:minimumlambda2} to obtain
	\begin{alignat*}{5} \label{eq:minimumlambdapblpa2}
		& \min        \quad && \lambda && && && \\
		& \text{s.t.} \quad && Pw &&\geqq q,  && &&\\
		&             \quad && w_1  + w_2 &&= \frac{1}{1+\lambda},&& && \tag{$\mathcal{P}^2(\lambda)$} \\
		&             \quad && \lambda \geq 0, && w \in \mathbb{R}_{\geqq}^{m+3}.&&&&
	\end{alignat*}
	%The weight $(0,0,1) \in \w(\tolp)$ is excluded for the same reason as described in Theorem~\ref{thm3} and Remark~\propitemref{remark2}{remark2:partone}. 
	%The weight $(0,0,1) \in \w(\tolp)$ is excluded because, based on the construction, the parameter $\lambda$ is undefined if $w_3^* = 1$. 
	However, we reformulate the above program to the following;
	\begin{alignat*}{5} \label{eq:minimumlambdapblp2}
		& \max        \quad && w_1  + w_2 && && && \\
		& \text{s.t.} \quad && Pw \geqq q && && && \tag{$\mathcal{P}^2_\wsc(\lambda_\ell)$} \\
		&             \quad &&  w \in \mathbb{R}_{\geqq}^{m+3}.&& &&&&
	\end{alignat*}
	by incorporating the constraint $w_1  + w_2 = \frac{1}{1+\lambda}$ directly into the objective function, thereby removing the variable $\lambda$. This is valid because the function $\frac{1}{1+\lambda}$ is strictly decreasing for $\lambda \geq 0$ which means minimizing $\lambda$ is equivalent to maximizing $w_1  + w_2 $. 
	After finding the maximum value, say $s_\mathrm{max}$, the corresponding value of $\lambda_\ell$ is calculated as $\lambda_\ell = \frac{1}{s_\mathrm{max}}-1$.\\
	Subsequently, we solve the minimization variant to determine $\lambda_u$, as illustrated in Figure~\ref{fig10:a}.
	We observe that the optimal value of this program can be $0$, which occurs because the maximization variant of the program~\ref{eq:minimumlambdapblpa2} is unbounded. This implies that the parameter value is approaching infinity.
	Therefore, we state that the corresponding solution remains optimal over the parameter interval $[\lambda_\ell, \infty)$ in such cases.
	%Note that the maximization variant of this program can be unbounded, particularly as $\lambda \to \infty$. When the optimal value is $0$, this implies that the parameter value approaches infinity. 

	\begin{figure*}[h]
		\begin{subfigure}[b]{0.49\textwidth}
			\centering
			\begin{tikzpicture}[scale=0.8]
				\draw [-Stealth] (0,0) -- (0,6) node[above left]{$w_2$};
				\draw [-Stealth] (0,0) -- (6, 0) node[below right]{$w_1$};
				\draw [fill=green, fill opacity = 0.25,ultra thin, dash pattern=on 1pt off 2pt] plot coordinates {(1,0) (1,3/2) (35/10,0)};
				\node at (1,-0.2) {$\lambda_u$};
				\node at (35/10,-0.2) {$\lambda_\ell$};
				\node at (2.1,0.3) {\tiny $\w(y^4)$};
				\draw [thick, blue!70] plot coordinates {(0,35/10) (35/10, 0)} node[above] {\tiny $ \mathcal{L}^2_{\lambda_\ell}$};
				\draw [thick, violet] plot coordinates {(0, 1)(1,0)} node[above left, xshift=4.5pt, yshift=7pt] {\tiny $\mathcal{L}^2_{\lambda_u}$};
			\end{tikzpicture}
			\caption{}
			\label{fig10:a}
		\end{subfigure}
		\hfill
		\begin{subfigure}[b]{0.49\textwidth}
			\centering
			\begin{tikzpicture}[scale=0.8]
				\draw [-Stealth] (0,0) -- (0,6) node[above left]{$w_2$};
				\draw [-Stealth] (0,0) -- (6, 0) node[below right]{$w_1$};
				\draw [fill=green, fill opacity = 0.25,ultra thin, dash pattern=on 1pt off 2pt] plot coordinates {(1,0) (1,3/2) (35/10,0)};
				\draw [thick, blue!70] plot coordinates {(0,5) (35/10,0)} node[above] {\tiny $\mathcal{L}^1_{\lambda_\ell}$};
				\draw [thick, violet] plot coordinates {(0,5) (1,0)} node[above left] {\tiny $\mathcal{L}^1_{\lambda_u}$};
				\node at (1,-0.2) {$\lambda_u$};
				\node at (35/10,-0.2) {$\lambda_\ell$};
				\node at (2.1,0.3) {\tiny $\w(y^4)$};
			\end{tikzpicture}
			\caption{}
			\label{fig10:b}
		\end{subfigure}
		\caption{%Weight set component $\w(y^1)$ with two line segments $\mathcal{L}^1_{\lambda_\ell}$,$\mathcal{L}^1_{\lambda_u}$ and $\mathcal{L}^2_{\lambda_\ell}$,$\mathcal{L}^22_{\lambda_u}$ representing $\mathcal{L}_{\w(\pblp^j(\lambda),w)}$ for $\lambda = \lambda_\ell, \lambda_u$ with respect to  $\pblp^2$ (\ref{fig10:a}) and  $\pblp^1$ (\ref{fig10:b}).}
		Weight set component $\w(y^1)$ with line segments $\mathcal{L}^j_{\lambda_\ell}$ and $\mathcal{L}^j_{\lambda_u}$ for $\pblp^2$ (\ref{fig10:a}) and $\pblp^1$ (\ref{fig10:b}), where $\mathcal{L}^j_\lambda = \mathcal{L}_{\w(\pblp^j(\lambda))}$ for $\lambda \in \left\{\lambda_\ell, \lambda_u\right\}$.}
	\label{fig10}
	\end{figure*}
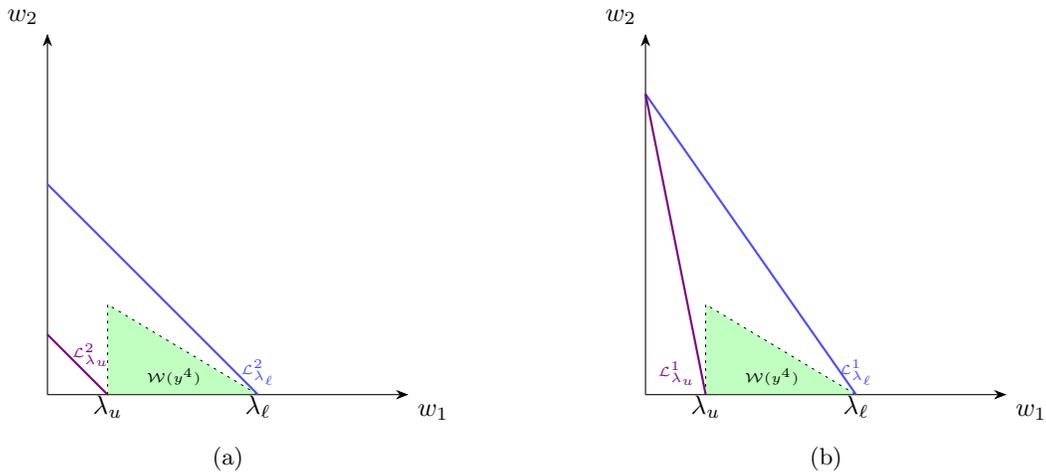

	Similarly for $\pblp^1$, we substitute the description of the line segment, $\mathcal{L}_{\w(\pblp^1(\lambda))}$ from Remark~\propitemref{remark1}{remark1:parttwo} into the program~\ref{eq:minimumlambda1} to get
	\begin{align*}    
	& \begin{alignedat}{5} \label{eq:minimumlambdapblp1}
		& \mathrm{min} \quad & \lambda \\
		& \text{s.t.} &  Pw &\geqq q, \\
		&&   w_1(1 + \lambda) + w_2 & = 1,\\
		&& (w_1,w_2) & \neq (0,1)
	\end{alignedat} \\
	& \hphantom{\min\quad}w \in \mathbb{R}_{\geqq}^{m+3},\lambda\geq0.
	\end{align*}
	%The weight $(0,1,0) \in \w(\tolp)$ is excluded in this program for the same reason as described in Theorem~\ref{thm2} and Remark~\propitemref{remark1}{remark1:partone}.
	%approaching $0$ is equivalent to the parameter $\lambda$ approaching $\infty$.
	In this program, we cannot apply the direct approach used in $\pblp^2$ because in constraint $w_1(1 + \lambda) + w_2 = 1$, $\lambda$ is not additively separable from the weight variables. We therefore normalize the constraint by dividing by $2+\lambda$ to obtain:
	\begin{align*}
	\left(\frac{1+\lambda}{2+\lambda}\right) w_1 + \left(\frac{1}{2+\lambda} \right) w_2= \frac{1}{2+\lambda}.
	\end{align*}
	We then define $\ell_1$ and $\ell_2$ as
	\begin{align*}
	\ell_1 \coloneqq \frac{1+\lambda}{2+\lambda}; \quad \ell_2 \coloneqq \frac{1}{2+\lambda}
	\end{align*}	
	and reformulate the problem as
	\begin{align*}    
	& \begin{alignedat}{4}
		& \mathrm{min} \quad & \ell_1 \\
		& \text{s.t.} &  Pw &\geqq q, \\
		&&   \ell_1 w_1 + \ell_2 w_2 & = \ell_2, \\
		&&    \ell_1 + \ell_2 & = 1, \\
		&& (w_1,w_2) & \neq (0,1),
	\end{alignedat}\tag{$\mathcal{P}$} \\
	& \hphantom{\min\quad} w \in \mathbb{R}_{\geqq}^{m+3}, \ell \in \mathbb{R}_{\geqq}^{2}.
	\end{align*}
	Crucially, because $\ell_1 = \left(\frac{1+\lambda}{2+\lambda}\right)$ is a strictly increasing function of $\lambda \geq 0$, the objective of minimizing $\lambda$ is equivalent to minimizing $\ell_1$.
	However, the constraint $\ell_1w_1 +\ell_2w_2 =\ell_2$ in $\mathcal{P}$ remains non-linear, as does the exclusion of $(0,1,\mathbf{0})$. We overcome these problems by solving the linear program
	\begin{align*}
	& \begin{alignedat}{4}
		& \max \quad & \ell_1 & & & \\
		& \text{s.t.} &P^\top u & & &\geqq \begin{pmatrix}
			\ell_1 \\
			\ell_2 \\
			\mathbf{0}
		\end{pmatrix},  \\
		&& q^\top u & {}-{} &\ell_2 & = 0,\\
		&&  \ell_1 & {}+{} &\ell_2 & = 1, 
	\end{alignedat} \tag{$\mathcal{P}^1_\wsc(\lambda_\ell)$} \\
	& \hphantom{\max \quad}  u \in\mathbb{R}_{\leqq}^{n+4}, \ell \in \mathbb{R}_{\geqq}^{2},
	\end{align*}
	instead of $\mathcal{P}$,
	where $\mathbf{0}$ is the zero vector with $m+1$ entries.. 
	%The equivalence of the optimal function value of $\mathcal{P}$ and $\mathcal{P}^1_\wsc(\lambda)$ is shown in Proposition~\ref{prop7}. 

	\begin{proposition}
	Let $\ell_1^*$ be the optimal function value of $\mathcal{P}$. Then $\ell_1^*$ is also the optimal function value of the linear program $\mathcal{P}^1_\wsc(\lambda_\ell)$. \label{prop::optimality}
	\end{proposition}
	\begin{proof}
	In this proof, we use the following pair of dual linear programming problems \\
	\begin{minipage}[t]{0.48\textwidth} 
		\begin{alignat*}{2}
			&\mathrm{\max} \ \quad && \ell_1w_1  +\ell_2w_2\\ \tag{$L_{\max}(\ell)$}
			& \mathrm{s.t.} && Pw \geqq q\\ 
			&  &&  w \in \mathbb{R}_{\geqq}^{m+3}
		\end{alignat*}
	\end{minipage}%
	\hfill 
	\begin{minipage}[t]{0.48\textwidth} 
		\begin{alignat*}{2}
			&\mathrm{\min} \ \quad &q^\top&  u\\ \tag{$L_{\min}(\ell)$}
			& \mathrm{s.t.} &P^\top&  u \geqq \begin{pmatrix}
				\ell_1\\
				\ell_2 \\
				\mathbf{0}
			\end{pmatrix} \\ 
			& \quad &&  u \in \mathbb{R}_{\leqq}^{n+4}.\\
		\end{alignat*}
	\end{minipage}
	Let $(\ell^*, w^*)$ be a feasible solution of $\mathcal{P}$ with solution value $\ell_1^*$.
	The proof is split into two parts; first, we show that there is a corresponding feasible solution $(\ell^*, u^*)$ of $\mathcal{P}^1_\wsc(\lambda_\ell)$.
	Second, we show that $(\ell^*, u^*)$ is also an optimal solution of $\mathcal{P}^1_\wsc(\lambda_\ell)$.
	
	\begin{enumerate}[label=\roman*.]
		\item Feasibility\\
		Since $(\ell^*,w^*)$ is feasible for $\mathcal{P}$, 
		\begin{align*}
			\ell_1^*w_1^* + \ell_2^*w_2^* = \ell_2^*. \tag{$\ast$} \label{eq:constraint1}
		\end{align*}
		Clearly, $w^*$ is also feasible for $L_{\max}(\ell^*)$ and has an objective function value $\ell_2^*$. \\
		Suppose $w^*$ is not optimal for $L_{\max}(\ell^*)$, then there exists some $w'$ that is feasible for  $L_{\max}(\ell^*)$ and\\
		\begin{align*}
			\ell_1^* w_1' + \ell_2^*w_2' > \ell_1^*w_1^* + \ell_2^*w_2^* &= \ell_2^*. \\
			\intertext{This implies}
			\ell_1^* w_1' + \ell_2^*(w_2'-1) &> 0.\\
			\intertext{Since $\ell_2^* = 1 - \ell_1^*$,} \label{eq:nonnegative}
			\ell_1^*(1+ w_1'-w_2') + w_2'-1 &>0. \tag{$\ast\ast$}
		\end{align*}
		Moreover, we have $(w_1,w_2) \neq (0,1)$ and $w \in \mathbb{R}^{m+3}_\geqq$. This means that the terms $1+ w_1'-w_2'$ and $w_2'-1$ in Inequality~(\refeq{eq:nonnegative}) are positive and negative, respectively. Thus, there exists some $\ell_1' < \ell_1^*$ such that 
		\begin{align*}
			\ell_1^*(1+ w_1'-w_2') + w_2'-1 &> \ell_1'(1+ w_1'-w_2') - w_2'+1 = 0 .\\
			\intertext{Using $\ell_2' = 1 - \ell_1'$, this implies}
			\ell_1' w_1' + \ell_2'w_2' &= \ell_2'.
		\end{align*}
		This contradicts our assumption that $(\ell^*,w^*)$ is an optimal solution of $\mathcal{P}$ since $(\ell',w')$ is also feasible and achieves a better solution value.
		Thus, $w^*$ is optimal for $L_{\max}(\ell^*)$.
		
		Since $L_{\max}(\ell^*)$ is feasible and has an optimal solution, the dual $L_{\min}(\ell^*)$ is feasible and bounded. Due to strong duality,  there exists some $u^*$ that is feasible for $L_{\min}(\ell^*)$ such that 
		\begin{align*}
			q^\top u^* = \ell_1^*w_1^* +\ell_2^*w_2^*.
		\end{align*}
		This leads to
		\begin{align*}
			q^\top u^*  = \ell_2^*.
		\end{align*}
		Thus, $(\ell^*, u^*)$ is also feasible for $\mathcal{P}^1_\wsc(\lambda_\ell)$.\\
		\item Optimality\\
		Assume $\ell_1^*$ is not the optimal function value of $\mathcal{P}^1_\wsc(\lambda_\ell)$. Then, there exists $\tilde{\ell}_1$ such that $\tilde{\ell}_1 < \ell_1^*$. The corresponding optimal solution $(\tilde{\ell}, \tilde{u})$ is feasible for $\mathcal{P}^1_\wsc(\lambda_\ell)$ and, in particular, it satisfies
		\begin{align*}
			q^\top \tilde{u}  = \tilde{\ell_2}.
		\end{align*}
		Clearly, $\tilde{u}$ is feasible for $L_{\min}(\tilde{\ell})$.
		At the same time, $w^*$ is feasible for $L_{\max}(\tilde{\ell})$ and due to weak duality, 
		\begin{align*}
			\tilde{\ell} w_1^* + \tilde{\ell_2} w_2^* \leq q^\top \tilde{u} = \tilde{\ell_2}.
		\end{align*}
		However, from Equality~(\refeq{eq:constraint1}), we obtain
		\begin{align*}
			\ell_1^*w_1^* +\ell_2^*w_2^* &=  \ell_2^*.\\
			\intertext{Using $\ell_2^* = 1 - \ell_1^*$, we get}
			\ell_1^*(1+ w_1^*-w_2^*) + w_2^*-1 &=0.\\
			\intertext{Since $\tilde{\ell}_1 > \ell_1^*$ and due to the same reasoning for Inequality~(\ref{eq:nonnegative}),} 
			\tilde{\ell}_1(1+ w_1^*-w_2^*) + w_2^*-1& >0\\
			\implies \tilde{\ell}_1 w_1^* + \tilde{\ell}_2 w_2^* &> \tilde{\ell_2}.
		\end{align*}
		This is not possible due to weak duality, and we get a contradiction. Thus, $\ell_1^*$ is an optimal function value of $\mathcal{P}^1_\wsc(\lambda_\ell)$.
	\end{enumerate}
	\end{proof}
	%Next, using the minimum value of $\mathcal{P}^1_\wsc(\lambda)$, say $\ell^*_1$, we compute the corresponding parameter value $\lambda$ with the following equation;
	%\begin{align*}\label{eqn12}
	%	\lambda_\ell = \frac{w_3^*}{w_1^*} \tag{12}
	%\end{align*}
	After finding the maximum value of $\mathcal{P}^1_\wsc(\lambda_\ell)$, say $\ell^*_1$, the corresponding value of $\lambda_\ell$ is calculated as $\lambda_\ell = \frac{1-2 \ell^*_1}{1-\ell^*_1}$. Observe that the optimal value of $\ell_1 = 1$ in $\mathcal{P}^1_\wsc(\lambda_\ell)$ implies $\lambda \to \infty$ in the parameter set.
	
	Using a similar argument, we solve the following minimization variant of $\mathcal{P}^1_\wsc(\lambda_\ell)$ to compute the minimum $\ell_1$ that corresponds to the parameter $\lambda_u$ 
	\begin{align*}
	& \begin{alignedat}{4}
		& \min        \quad & \ell_1  & & &            \\
		& \text{s.t.} \quad & P^\top u & & & \leqq \begin{pmatrix} \ell_1 \\ \ell_2 \\\mathbf{0} \\ \end{pmatrix}, \\
		&& q^\top u  & {}-{} &\ell_2 & = 0, \\
		&&  \ell_1 & {}+{} &\ell_2 & = 1,  
	\end{alignedat} \tag{$\mathcal{P}^1_\wsc(\lambda_u)$} \\
	& \hphantom{\max \quad}   u \in \mathbb{R}_{\geqq}^{n+4}, \ell \in \mathbb{R}_{\geqq}^{2}.
	\end{align*}
	
	%Finally, we use the following equation, 
	%\begin{align*}\label{eq:lambdavalue}
	%	\lambda \coloneqq \frac{\ell_1}{\ell_2}-1 \tag{11}
	%\end{align*}
	%to find the corresponding $\lambda$ values. This equation (\ref{eq:lambdavalue}) is a direct derivation from the fact that the weight sets of $\pblp^1$ are hyperplanes $-\ell_1 w_1 + \ell_2w_2 = -\ell_2$ with the corresponding slope $\frac{-\ell_1}{\ell_2}$ and Remark~\propitemref{remark1}{remark1:parttwo}. 
	%Moreover, the optimal function value $\ell_1^*$ of $\mathcal{P}^1_\wsc(\lambda)$ and $\mathcal{P}_{w_l}$ corresponds to the slope of the lines that represent the weight sets of $\pblp^1(\lambda_u)$ and $\pblp^1(\lambda_\ell)$, respectively.\\
	
	As we have simplified the weight set component in equation~(\ref{eqn:10}), we now incorporate the formal definition of the weight set component into the linear program $\mathcal{P}^1_\wsc(\lambda_\ell)$. 
	Consequently, we solve the following linear program to find $\lambda_\ell$:
	\begin{align*}
	\begin{alignedat}{6}
		& \mathrm{max} \quad & && && && \ell_1\\
		& \mathrm{s.t.} & Ax & {}+{} & b x_{\mathrm{opt}} & & & & &\leqq0, \\
		&& -Cx & {}-{} & yx_{\mathrm{opt}} & {}+{} & x_w & {}-{} & \begin{pmatrix}
			-1 & 0 \\
			0  & -1 \\
			0  & 0 
		\end{pmatrix} \text{\makebox[0pt][l]{$\ell$}\hphantom{$\ell_2$}} & \leqq 0, \\
		&& && && x_w & {}-{} &\ell_2 & = 0,\\
		&& && && && \ell_1 + \ell_2 & = 1,
	\end{alignedat}\\
	x \in \mathbb{R}_{\geqq}^n,\ x_{\mathrm{opt}} \in\mathbb{R},\ x_w \in \mathbb{R},\ \ell \in \mathbb{R}_{\geqq}^2.
	\end{align*}
	
	This procedure is applied to each solution $x \in S$, generating the corresponding parameter interval $\left[\lambda_\ell, \lambda_u\right]$. The union of all interval boundaries constitutes the breakpoints set, with the complete algorithm presented in Algorithm~\ref{alg1}.
	
	\begin{algorithm}[t]
	\caption{Breakpoint enumeration algorithm}
	\label{alg1}
	\begin{algorithmic}[1]
		\Require A minimal solution set $S$ that corresponds to the extreme nondominated images $Y_\en$ of the corresponding TOLP.
		\Ensure For all $x \in S$,  a parameter interval $[\lambda_\ell, \lambda_u]$ and a set of breakpoints $\mathcal{B}$.
		
		%\State Initialize $\mathcal{B} \gets \{\}$ \Comment{Stores all computed $\lambda_\ell, \lambda_u$ as breakpoint candidates}
		%\State Initialize $\mathcal{D} \gets \{\}$ \Comment{Key: $\lambda$ value, Value: set of optimal solutions at this $\lambda$}
		%\State Initialize $\mathcal{I} \gets \{\}$ \Comment{Stores the final intervals and their optimal sets}
		%\If {j = 1}
		\For{\textbf{all} $x\in S$}
		\State $\lambda_\ell \gets$ Solve $\mathcal{P}^j_\wsc(\lambda)$
		\State $\lambda_u \gets $ Solve the opposite variant of $\mathcal{P}^j_\wsc(\lambda)$
		%\State Solve the linear programs minimize $\lambda$ and maximize $\lambda$ to find $\lambda_\ell$ and $\lambda_u$ respectively.
		%\State \Comment{Now compute the corresponding $\lambda$ values}
		%\State $\lambda_\ell^{(k)} \gets \frac{\ell_1}{\ell_2} - 1 $ 
		\State Add $\lambda_\ell^{(k)}$ and $\lambda_u^{(k)}$ to the breakpoint list $\mathcal{B}$
		%\State For all $\lambda$ in $[\lambda_\ell^{(k)}, \lambda_u^{(k)}]$, $\mathbf{y}^{(k)}$ is optimal. Store this in $\mathcal{D}$.
		\EndFor
		%\ElsIf {j = 2}
		%\For{\textbf{all} $\mathbf{y}^{(k)} \in Y_\en$}
		%\State Solve $P_\wsc$ using the weights $\ell \coloneqq (1,1)$ and $-\ell$ to compute $w_u$ and $w_l$ respectively.
		%\State \Comment{Now compute the corresponding $\lambda$ values}
		%\State $\lambda_u^{(k)} \gets\frac{\mathbf{w}_3^u}{1- \mathbf{w}_3^u}$ 
		%\State $\lambda_\ell^{(k)} \gets \frac{\mathbf{w}_3^l}{1- \mathbf{w}_3^l} $
		%\State Add $\lambda_\ell^{(k)}$ and $\lambda_u^{(k)}$ to the breakpoint list $\mathcal{B}$
		%\State For all $\lambda$ in $[\lambda_\ell^{(k)}, \lambda_u^{(k)}]$, $\mathbf{y}^{(k)}$ is optimal. Store this in $\mathcal{D}$.
		%\EndFor
		%\EndIf
		%\State \textbf{Sort} the list of breakpoints $\mathcal{B}$ in ascending order and remove duplicates
		%\State \Comment{Now, process the sorted breakpoints to create the intervals}
		%\For{$i \gets 0$ \textbf{to} $|\mathcal{B}| - 2$} \Comment{Iterate over consecutive pairs}
		%\State $\lambda_{\text{start}} \gets \mathcal{B}[i]$
		%\State $\lambda_{\text{end}} \gets \mathcal{B}[i+1]$
		%\State \todo{Choose a representative $\lambda^*$ in the open interval $(\lambda_{\text{start}}, \lambda_{\text{end}})$ (e.g., the midpoint)}
		%\State Find the set $S^*$ of all extreme nondominated  points $\mathbf{y}^{(k)} \in Y_\en$ for which $\lambda_\ell^{(k)} \leq \lambda^* \leq \lambda_u^{(k)}$
		%\State \textbf{Append} the tuple $\Big( [\lambda_{\text{start}}, \lambda_{\text{end}}],\ S^* \Big)$ to the interval list $\mathcal{I}$
		%\EndFor
	\end{algorithmic}
	\end{algorithm}
	If necessary, the breakpoints can be sorted using the collection of parameter intervals obtained for all $x \in S$. Each resulting parameter interval (or a particular value) in the parameter set corresponds to a unique minimal solution set, with breakpoints marking the critical parameter values at which transitions between solution sets occur. Specialized data structures, for example interval trees, can be used to optimize this sorting and partitioning process (cf.~\cite{Edelsbrunner1983}).
	%Thats the end of the algorithm but if necessary, we can sort the breakpoints using the collection of parameter intervals for all $x \in S$. Each resulting interval (or value) of the parameter set corresponds to a unique minimal solution set, with breakpoints identifying the critical parameter values where the solution sets transition. Specialized data structures can be used to optimize this sorting and partitioning procedure.
	%We sort the breakpoints in order to divide the parameter set into intervals and/or values which corresponds to a unique minimal solution set. Using these intervals of $\lambda$ for each $ x \in S$ to find breakpoints such that we have a change in minimal solution sets at these breakpoints. 
	%We create a set of tuples $\mathcal{I}$ defined as $\mathcal{I} \coloneqq \{([\lambda_\ell, \lambda_u], S^*)\colon \lambda_\ell, \lambda_u \in \mathcal{B}, S^* \subseteq S \}$ to store the final intervals and their optimal sets.  
	
	\begin{theorem}
	The algorithm finds all the breakpoints in the parameter set.
	\end{theorem}
	\begin{proof}
	The maximum and minimum values of the parameter interval  dictate when a feasible solution $x$ enters and exits a minimal solution set $S$. 
	Thus, leading to a change in a minimal solution set and serves as breakpoints.
	%	The algorithm traverses over all the weight set components of all extreme nondominated images. By Proposition~\ref{prop7}, we know that all breakpoints correspond to some extreme point of the weight set components. Therefore, for each breakpoint, there is a corresponding minimal solution set with respect to the BOLP $\pblp$ at that breakpoint. 
	\end{proof}
	
	\begin{theorem}
	The algorithm has a running time of $\mathcal{O}{\left(\left| S\right| T_\mathrm{ws}\right)}$, where $T_\mathrm{ws}$ is the running time of the linear program $\mathcal{P}^j_\wsc(\lambda)$.
	\end{theorem}
	\begin{proof}
	The  minimization and maximization variants of the linear program $\mathcal{P}^j_{\wsc}(\lambda)$ can be solved polynomially in the encoding size of the program $\pblp^j$, say $T_\mathrm{ws}$ (cf.\cite{Boekler2015}). Other evaluations such as computing the corresponding parameter value and sorting also take polynomial time. The algorithm iterates over all the extreme nondominated images in $Y_\en$, thus it has a running time of $\mathcal{O}{\left(\left| S\right| T_\mathrm{ws}\right)}$.
	\end{proof}
	
	\subsection{Adapted weight set algorithm}
	As shown in the breakpoint enumeration algorithm (Algorithm~\ref{alg1}), the weight set components of $\w(\tolp)$ can be used to solve $\pblp^j$. 
	We now discuss a strategy that directly utilizes weight set decompositions of $\tolp$ to address our problem. 
	Weight set decomposition algorithms compute all extreme nondominated images of multiobjective linear programs, with the weight set decomposition emerging as an auxiliary output. 
	
	The algorithm begins with an approximation of the weight set components. It iteratively shrinks this approximation by verifying the vertices of the current weight set component until the exact weight set components are obtained.
	This approach ensures that all extreme points are identified for every weight set component associated with each extreme nondominated image. 
	We use these extreme weights to calculate the parameter intervals for all weight set components.
	%As we have shown in Algorithm \ref{alg1}, the weight set components of $\w(\tolp)$ can be used to solve $\pblp^j$. Weight set decomposition algorithms compute all extreme nondominated points of multiobjective programs. The decomposition of the weight set emerges as a byproduct during this computation, ensuring that all extreme points of the weight set components corresponding to nondominated extreme points are identified. In this algorithm, we adapt any of the existing weight set decomposition algorithms which ensure the computation of all the weight set components along with its extreme points for all $y \in Y_\en(\tolp)$. 
	The corresponding parameter values are computed using Equation~(\ref{eqn3}) from Theorem~\ref{thm1} for $\pblp^1$ and Equation~(\ref{eqn5}) from Theorem~\ref{thm2} for $\pblp^2$ as shown in Figure~\ref{fig12}.
	
	We begin with an empty parameter interval. 
	The first two parameters are used to initialize the upper and lower bounds, creating the initial interval $[\lambda_\ell, \lambda_u]$. Subsequently, for each new vertex processed,  the parameter is evaluated: if it is greater than $\lambda_u$, the upper bound is updated; if it is less than $\lambda_\ell$, the lower bound is updated.
	Thus, we determine the parameter interval $[\lambda_\ell, \lambda_u]$ within which the corresponding $x \in S$ of $y$ remains optimal.
	These parameter bounds are added to a set of breakpoints, if they are not already in the set.
	We repeat this process for every weight set component of $y \in Y_\en(\tolp)$.

	%Finally, we use this information to find the set of breakpoints and its corresponding minimal solution set.
	
	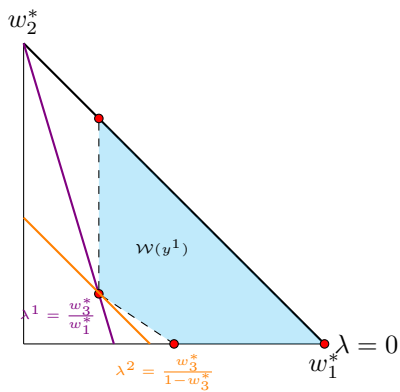
\begin{figure*}[]
	\centering
	\begin{tikzpicture}[scale=0.8]
		\draw (0,0) -- (5,0) node[below]{$w_1^*$};
		\draw (0,0) -- (0,5) node[above]{$w_2^*$};
		\draw plot coordinates {(0,5) (5, 0)};
		\draw [thick] plot coordinates{(0,5) (5, 0)} node[right] {$\lambda = 0$};
		\draw [thick, violet] plot coordinates {(0,5) (1.5, 0)} node[above left,xshift=-0.1cm,yshift=0] {\tiny $\lambda^1  = \frac{w_3^*}{w_1^*}$};
		\draw [dashed, fill=cyan, fill opacity = 0.25] plot coordinates {(2.5,0) (1.25, 5/6) (1.25, 3.75) (5, 0)};
		\draw plot[mark=*, mark options={fill=red}] coordinates {(2.5,0)};
		\draw plot[mark=*, mark options={fill=red}] coordinates {(5, 0)};
		\draw plot[mark=*, mark options={fill=red}] coordinates {(1.25, 5/6)};
		\draw plot[mark=*, mark options={fill=red}] coordinates {(1.25, 3.75)};
		\draw [thick, orange] plot coordinates {(0,2.1) (2.1, 0)} node[below, xshift=0.2cm] {\tiny $\lambda^2 = \frac{w_3^*}{1- w_3^*}$};
		\node at (2.3,1.6) {\tiny $\w(y^1)$};
	\end{tikzpicture}
	\caption{Evaluation of $\lambda^j$ associated to $\w(y^1)$ for $\pblp^j$, $j =1,2$. For all the extreme points of $\w(y^1)$ marked in red, its corresponding $\lambda$ values are computed. A solution $x^1 \in S$ is optimal for the parameter intervals $[0, \lambda_1]$ in $\pblp^1$ and $[0, \lambda_2]$ in $\pblp^2$. \label{fig12}}
	\end{figure*}
	
	%The computational overhead of determining breakpoints is minimal, as it involves only the calculation of parameter values and comparisons against the current interval bounds. 
	Parameter value calculations and comparisons with current interval bounds require $\mathcal{O}(1)$ time per operation in addition to the time taken to enumerate each extreme weight of a weight set component. Consequently, this constant-time overhead is asymptotically negligible in the overall weight set decomposition algorithm.
	This contrasts to the more significant overhead of the Breakpoint Enumeration Algorithm~\ref{alg1}, which is bounded by 
	$\mathcal{O}{\left(\left| S\right| T_\mathrm{ws}\right)}$. 
	A notable limitation of this method, however, is its dependence on a weight set decomposition algorithm, which restricts the choice of the underlying multi-objective programming algorithms. However, the Breakpoint Enumeration Algorithm~\ref{alg1} itself is efficient in scenarios where it computes the set $Y_\en(\tolp)$ faster than a complete weight set decomposition.
	
	%The algorithm computes all the weight set components of all the extreme nondominated images of the TOLP. We use the extreme points of the weight set component to compute the corresponding parameter values that are all candidates for breakpoints. Moreover, by Proposition \ref{prop6}, one breakpoint corresponds to at least one extreme weight. Thus, the algorithm computes all breakpoints and for each breakpoint $\lambda$, we have a unique solution set for $\pblp^j(\lambda)$. 
	
	%The running time of this algorithm will be dominated by the weight set decomposition algorithm because computing successive parameter values using extreme points of the weight set components can be done in polynomial time. 
	
	\section{Conclusion}\label{sec5}
	This article relies on the existing multi-objective programming methods and the weight set decomposition of the triobjective linear program to solve the parametric biobjective linear programs. 
	Our approach uses the weighted sum scalarization of the parametric biobjective linear program and the corresponding triobjective linear program formulation. 
	We have developed a theoretical framework relating the solution set of the parametric program to a multi-objective linear program using the structure of weight sets of both these problems.
	We propose two algorithms to address two distinct cases of parametric biobjective linear programs. 
	The first algorithm builds upon any multi-objective linear programming algorithm and solves linear programs to find breakpoints in the parameter set.
	For the second algorithm, we specifically adapt a weight set decomposition algorithm to enumerate the breakpoints. 
	
	\section*{Acknowledgments}
	Kezang Yuden is supported by the German Academic Exchange Service (DAAD) under the funding programme "Development-Related Postgraduate Course" (ID-57610012).
	Levin Nemesch is funded by the Deutsche Forschungsgemeinschaft (DFG, German Research Foundation) -- Project number 508981269 and GRK 2982, 516090167 ``Mathematics of Interdisciplinary Multiobjective Optimization``.
	
	\bibliographystyle{unsrt}
	\bibliography{Bibliography}
	
\end{document}